\documentclass[final,onefignum,onetabnum]{siamltex1213}

\usepackage{amsmath}               
  {
      \newtheorem{assumption}{Assumption}
      \newtheorem{remark}{Remark}
  }
\usepackage{amsfonts}
\usepackage{booktabs}
\usepackage[tworuled]{algorithm2e}
\usepackage{subfiles}
\usepackage{appendix}

\title{Kalman-based Stochastic Gradient Method with Stop Condition and Insensitivity to Conditioning} 

\author{Vivak Patel \thanks{University of Chicago, Department of Statistics. 
(\email{vp314@uchicago.edu}). Questions, comments, or corrections
to this document may be directed to that email address.}}

\newcommand{\norm}[1]{\left\Vert #1 \right\Vert}
\newcommand{\1}[1]{\textbf{1}\left[ #1 \right]}
\newcommand{\Prb}[1]{\mathbb{P}\left[ #1 \right]}
\newcommand{\E}[1]{\mathbb{E}\left[ #1 \right]}
\newcommand{\tr}[1]{\mathbf{tr}\left[  #1 \right]}
\newcommand{\bigO}[1]{\mathcal{O}\left[ #1 \right]}
\newcommand{\cond}[2]{\mathbb{E}\left[\left. #1 \right\vert #2 \right]}

\newcommand{\NchooseK}[2]{\genfrac{(}{)}{0pt}{}{#1}{#2}}

\begin{document}
\maketitle
\slugger{siopt}{xxxx}{xx}{x}{x--x}

\begin{abstract}
Modern proximal and stochastic gradient descent (SGD) methods are believed to efficiently minimize large composite objective functions, but such methods have two algorithmic challenges: (1) a lack of fast or justified stop conditions, and (2) sensitivity to the objective function's conditioning. In response to the first challenge, modern proximal and SGD methods guarantee convergence only after multiple epochs, but such a guarantee renders proximal and SGD methods infeasible when the number of component functions is very large or infinite. In response to the second challenge, second order SGD methods have been developed, but they are marred by the complexity of their analysis. In this work, we address these challenges on the limited, but important, linear regression problem by introducing and analyzing a second order proximal/SGD method based on Kalman Filtering (kSGD). Through our analysis, we show kSGD is asymptotically optimal, develop a fast algorithm for very large, infinite or streaming data sources with a justified stop condition, prove that kSGD is insensitive to the problem's conditioning, and develop a unique approach for analyzing the complex second order dynamics. Our theoretical results are supported by numerical experiments on three regression problems (linear, nonparametric wavelet, and logistic) using three large publicly available datasets. Moreover, our analysis and experiments lay a foundation for embedding kSGD in multiple epoch algorithms, extending kSGD to other problem classes, and developing parallel and low memory kSGD implementations.
\end{abstract}



\pagestyle{myheadings}
\thispagestyle{plain}
\markboth{Vivak Patel}{Kalman-based Stochastic Gradient Method}

\section{Introduction}
In the data sciences, proximal and stochastic gradient descent (SGD) methods are widely applied to minimize objective functions over a parameter $\beta \in \mathbb{R}^n$ of the form:
\begin{equation}
\label{eqn:composite objective}
\sum_{k=1}^N f_k(\beta), \quad f_i:\mathbb{R}^n \to \mathbb{R}, \quad \forall i=1,\ldots,N
\end{equation}
When the number of component functions, $N$, is large, proximal and SGD methods are believed to require fewer memory and computing resources in comparison to classical methods (e.g. gradient descent, quasi-Newton) \cite{bert2010, combettes2011}. However, proximal and SGD methods' computational benefits are reduced by two algorithmic challenges: (1) the difficulty of developing computationally fast or justified stop conditions, and (2) the difficulty of overcoming sensitivity to the conditioning of the objective function \cite{combettes2011, roux2012, fercoq2015, konevcny2013, konevcny2014, xiao2014, shalev2014, byrd2016, moritz2015}. 

In response to the challenge of developing computationally fast, justified stop conditions, many proximal and SGD methods guarantee convergence up to a user-specified probability by requiring multiple epochs --- full passes over all $N$ component functions --- to be completed, where the number of epochs increases as the user-specified probability increases \cite{roux2012, fercoq2015, konevcny2013, konevcny2014, xiao2014, shalev2014, byrd2016, moritz2015}. As a result, these proximal and SGD methods become inefficient when high communication costs are incurred for completing a single epoch, and such is the case when $N$ is very large, streaming or infinite \cite{bert2010}.

In response to the challenge of overcoming sensitivity to the conditioning of the objective function, second order SGD methods have been developed. Unfortunately, second order SGD methods are marred by the complexity of their analysis owing to the nonlinear dynamics of their second order approximations, the stochastic nature of the gradients, and the difficulty of analyzing tuning parameters \cite{amari2000,schraudolph2007,bordes2009,byrd2016}. For example, in the method of \cite{amari2000}, the second order estimate averages weighted new information with a previous second order estimate. However, the second order estimate initialization and weight parameter tuning is difficult in practice \cite{schraudolph2007,bordes2009}, and is not accounted for in the analysis \cite{amari2000}. In order to account for the tuning parameters, one successful second order SGD analysis strategy is an adapted classical Lyapunov approach \cite{byrd2016, moritz2015}. However, the Lyapunov approach falls short of demonstrating that these second order SGD methods are insensitive to the conditioning of the objective function. For example, the Lyapunov-based convergence theorem in \cite{byrd2016} suggests that the convergence rate depends on the conditioning of the Hessian of the objective function times the conditioning of the BFGS estimates. Such a Lyapunov-based theoretical guarantee offers no improvement over what we achieve with the usual SGD \cite{murata1998,bousquet2008}, and clash against what we intuitively expect from a second order method.  

In this paper, we progress towards addressing these gaps by generalizing proximal and SGD methods using the principles of the Kalman Filter \cite{kalman1960}; we refer to this generalization as Kalman-based stochastic gradient descent (kSGD). We analyze kSGD's properties on the class of linear regression problems, which we concede is a limited class, but is an inherently important class as discussed further in \S \ref{sec:Linear Regression Problem}. Moreover, we use a probabilistic approach to analyze kSGD's stochastic second order dynamics and tuning parameter choices, thereby avoiding the difficulties associated with the Lyapunov approach. As a consequence of our analysis, 
\begin{remunerate}
\item We justify a computationally fast stop condition for the kSGD method, and demonstrate its effectiveness on several examples, thereby addressing the first algorithmic challenge faced by proximal and SGD methods.
\item We prove that the kSGD method is insensitive to the conditioning of the objective function, thereby addressing the second algorithmic challenge faced by proximal and SGD methods.
\item We show that the kSGD method aymptotically recovers the optimal stochastic gradient estimator for the linear regression problem. 
\item We create a provably convergent, fast and simple algorithm for solving the linear regression problem, which is robust to the choice of tuning parameters and which can be applied to very large, streaming or infinite data sources.
\item We lay a foundation for embedding kSGD in multiple epoch algorithms, extending kSGD to other problem classes, and developing parallel and low memory kSGD implementations.
\end{remunerate}

We organize the rest of the paper as follows. In \S \ref{sec:Linear Regression Problem}, we formulate the linear regression problem, and discuss its importance. In \S \ref{sec:Optimal Estimator + kSGD}, we construct an optimal proximal/SGD method from which we derive kSGD. In \S \ref{sec:Analysis}, we prove that kSGD approaches the optimal proximal/SGD method, we prove that kSGD is insensitive to the conditioning of the problem, we analyze the impact of tuning parameters on convergence, and we construct an adaptive tuning parameter strategy. In \S \ref{sec:Algorithms}, we present a kSGD algorithm with a stop condition, and an adaptive tuning parameter selection algorithm. In \S \ref{sec:complexity}, we compare the computational and memory complexities of kSGD with the usual SGD,  and SQN \cite{byrd2016}. In \S \ref{sec:Empirical Example}, we support the mathematical arguments herein using three numerical examples: a linear regression problem using a large data set provided by the Center for Medicare \& Medicaid Services (CMS), a nonparametric wavelet regression problem using a large data set of gas sensor readings, and a logistic regression problem using a moderately sized data set of adult incomes. In \S \ref{sec:Conclusion}, we summarize this work.

\section{The Linear Regression Problem \& Its Importance}
\label{sec:Linear Regression Problem}
First, the linear regression problem is formulated, and then its importance is discussed. In order to formulate the linear regression problem, the linear model must be specified:
\begin{assumption} \label{assumption:model}
Suppose that $(X,Y),(X_1,Y_1),(X_2,Y_2),\ldots \in \mathbb{R}^n \times \mathbb{R}$ are independent, identically distributed,  and $\exists \beta^* \in \mathbb{R}^n$ such that:
$$
Y_i = X_i^T\beta^*  + \epsilon_i
$$
where $\epsilon_i$ are independent, identically distributed mean zero random variables with variance $\sigma^2 \in (0, \infty)$, and are independent of all $X_i$. 
\end{assumption}
\begin{remark}
The model does not assume a distribution for the errors, $\epsilon_i$; hence, the results presented will hold even if the model is misspecified with a reinterpretation of $\sigma^2$ as the limiting mean residuals squared. In addition, if the model has heteroscedasticity, the convergence of the kSGD parameter estimate to $\beta^*$ will still hold in the results below as long as the supremum over all variances is bounded.
\end{remark}

Informally, the linear regression problem is the task of determining $\beta^*$ from the data $(X_1,Y_1),(X_2,Y_2),\ldots$. To formalize this, the linear regression problem can be restated as minimizing a loss function over the data, which is ideally (see \cite{bean2013})
$$ 
d(X,Y,\beta) = \frac{1}{2}\left( Y - \beta^T X \right)^2
$$
up to a positive multiplicative constant and conditioned on $X$. Because the ideal linear regression loss, $d(X,Y,\beta)$ is a function of random variables $(X,Y)$, it is also a random variable. In order to simplify optimizing over the random variable $d(X,Y,\beta)$, the linear regression objective function is taken to be the expected value of $d(X,Y,\beta)$:
\begin{equation}
\label{eqn:objective function}
D(\beta) := \E{d(X,Y,\beta)} = D(\beta^*) + \frac{1}{2}(\beta-\beta^*)^TQ_*(\beta-\beta^*)
\end{equation}
where $Q_*= \E{XX^T}$. 

Thus, the linear regression problem is the task of minimizing $D(\beta)$. However, on its own, the linear regression problem is ill-posed for several reasons. First, despite the simple form of the linear regression objective, it cannot be constructed because the distribution of $(X,Y)$ is rarely known. To account for this, the linear regression objective function is replaced with the approximation
\begin{equation}
\label{eqn:objective function, approximate}
\hat{D}(\beta) = \frac{1}{N}\sum_{i=1}^N d(X_i,Y_i,\beta)
\end{equation}
which is exactly of the form (\ref{eqn:composite objective}). Second, the linear regression objective's Hessian, $Q_*$, may not be well-specified, that is, $Q_* \not\prec \infty$. One way to ensure that $Q_* \prec \infty$ is to require
$$\lambda_{\max}(Q_*) \leq \tr{Q_*} = \E{\tr{XX^T}} = \E{\norm{X}_2^2} < \infty$$
This requirement is collected in the next assumption:
\begin{assumption}
\label{assumption:X is in L2}
$X \in L^2$. That is, $\E{\norm{X}_2^2} < \infty$. 
\end{assumption}

For some results, Assumption \ref{assumption:X is in L2} is strengthened by:
\begin{assumption}
\label{assumption:X is Linf}
$X \in L^\infty$. That is, $\norm{X}_\infty < \infty$ almost surely.
\end{assumption}

Third, from an optimization perspective, the minimizer of the linear regression objective must satisfy second order sufficiency conditions (Theorem 2.4 in \cite{nocedal2006}); that is, $Q_*$ must be positive definite. This can be ensured by the weaker requirement
\begin{assumption}
\label{assumption:X spans Rn}
The linear span of the image of $X$ is $\mathbb{R}^n$. Specifically, for all unit vectors $v \in \mathbb{R}^n$, 
$ \Prb{ |X^T v| = 0 } < 1$.
\end{assumption}

To see how, suppose there is a unit vector $v \in \mathbb{R}^n$ such that $0 \geq v^T Q_* v$. Then
$ 0 \geq v^T Q_* v = \E{(X^T v)^2}$.
Hence, $X^T v = 0 $ almost surely, which contradicts Assumption \ref{assumption:X spans Rn}.

\begin{remark} We suspect that Assumption \ref{assumption:X spans Rn} can be weakened by allowing $X$ to be orthogonal to a subspace. As a result, $\beta^*$ would no longer be unique, but would be specified by a hyperplane. To account for this in the analysis, the error measure would need to be replaced by the minimum distance between the estimate and this hyperplane, and the estimation sequence would need to be considered on the lower dimensional manifold specified by the image space of $X$.
\end{remark}

At first glance, the linear regression problem seems to be a limited problem class, and one that has already been sufficiently addressed; however, solving modern, very large linear regression problems is still a topic under active research (e.g. \cite{fabregat2014,frank2015}). One example of very large linear regression problems comes from recent physics studies in which background noise models are estimated from large simulated data sets using maximum likelihood estimation \cite{atlas2012,jaranowski2012}, which can be recast as solving a sequence of very large linear regression problems \cite{wedderburn1974}. Another example comes from genome-wide association studies in which multiple very large linear regression problems arise directly \cite{teslovich2010,aulchenko2010}. For such linear regression problems, the high communication time of reading the entire data set renders multiple epoch algorithms impractical. One recourse is to apply the usual SGD; however, SGD will often stall before converging to a minimizer. Thus, kSGD becomes a viable alternative for solving such problems as it only requires a single pass through large data sets, and does not stall. These concepts are demonstrated on a linear regression problem on CMS medical claims payments in \S \ref{subsec:linear regression}.

Moreover, the linear regression problem not only encompasses the usual linear model of Assumption \ref{assumption:model}, but also the normal means models which include many nonparametric regression models (see \cite{nussbaum1996}, Ch. 7 in \cite{wasserman2006}). For example, the linear regression problem in (\ref{eqn:objective function, approximate}) includes B-spline regression, for which $X_i$ is replaced by a vector-valued function evaluated at $X_i$ \cite{racine2014}. The linear regression problem's applicability to normal means models is demonstrated on a non-parametric wavelet regression problem on gas sensor reading data in \S \ref{subsec:nonparametric wavelet regression}.

Finally, the linear regression problem analysis is an essential step in generalizing the kSGD theory to other problem classes. To be explicit, a common pattern in nonlinear programming is to model the objective function locally as a quadratic, and determine the next iterate by minimizing this local model \cite{bert1999,nocedal2006}. Indeed, for objective functions with an underlying statistical model, the local quadratic model can be formalized using the theory of local asymptotic normality (see Ch. 5 in \cite{van2000}, Ch. 6 in \cite{lecam2012}); thus, understanding the behavior of kSGD in a quadratic model is an essential step in extending the analysis of kSGD to other problem classes. This principle is demonstrated on a logistic regression problem for modeling adult income categories in \S \ref{subsec:logistic regression}.

Now that we have established the linear regression problem, we construct an optimal stochastic gradient estimator from which we will derive kSGD. 

\section{An Optimal Estimator \& kSGD}
\label{sec:Optimal Estimator + kSGD}

Let $\mathcal{F}_{k} = \sigma(X_1,\ldots,X_k)$, the $\sigma$-algebra of the random variables $X_1,\ldots,X_k$, and consider the following general update scheme
\begin{equation}
\label{eqn:General Update Scheme}
\beta_{k+1} = \beta_k + G_{k+1}(Y_{k+1} - \beta_k^T X_{k+1})
\end{equation}
where $G_{k+1}$ is a random variable in $\mathbb{R}^n$ and is measurable with respect to $\mathcal{F}_{k+1}$. Using Assumption \ref{assumption:model}, (\ref{eqn:General Update Scheme}) can be rewritten as
\begin{equation*}
\label{eqn:General Update Scheme w Model Assumption}
\beta_{k+1} = \beta_{k} - G_{k+1}X_{k+1}^T(\beta_k - \beta^*) + G_{k+1}\epsilon_{k+1}
\end{equation*}
We will choose an optimal $G_{k+1}$ in the sense that it minimizes the $l^2$ error between $\beta_{k+1}$ and $\beta^*$ given $\mathcal{F}_{k+1}$. Noting that $G_{k+1}$ is measurable with respect to $\mathcal{F}_{k+1}$, and using the independence, first moment and second moment properties of $\epsilon_k$:
\begin{align*}
&\E{\left. \norm{\beta_{k+1} - \beta^*}^2 \right\vert \mathcal{F}_{k+1}} \\
	&\quad= \tr{\cond{(\beta_{k} - \beta^*)(\beta_{k} - \beta^*)^T}{\mathcal{F}_{k}}} \\
	&\quad\quad- \tr{G_{k+1}X_{k+1}^T\cond{(\beta_{k} - \beta^*)(\beta_{k} - \beta^*)^T }{ \mathcal{F}_{k}}}\\
	&\quad\quad- \tr{\cond{ (\beta_{k} - \beta^*)(\beta_{k} - \beta^*)^T }{ \mathcal{F}_{k}} X_{k+1}G_{k+1}^T}\\
	&\quad\quad+ \tr{ G_{k+1}X_{k+1}^T\cond{(\beta_{k} - \beta^*)(\beta_{k} - \beta^*)^T  }{\mathcal{F}_{k}} X_{k+1}G_{k+1}^T }\\
		&\quad\quad+ \sigma^2 \tr{G_{k+1}G_{k+1}^T}
\end{align*}
We now write $\mathcal{M}_{k} = \E{\left. (\beta_{k} - \beta^*)(\beta_{k} - \beta^*)^T \right\vert \mathcal{F}_{k}}$, which gives:
\begin{align*}
\E{\left. \norm{\beta_{k+1} - \beta^*}^2 \right\vert \mathcal{F}_{k+1}} 
	& = \tr{ \mathcal{M}_k } 
	- \tr{ G_{k+1}X_{k+1}^T \mathcal{M}_k}
	- \tr{\mathcal{M}_k X_{k+1}G_{k+1}^T } \\
	&\quad + \tr{ G_{k+1}X_{k+1}^T \mathcal{M}_k X_{k+1}G_{k+1}^T } 
	+ \sigma^2 \tr{G_{k+1}G_{k+1}^T}
\end{align*}
Differentiating with respect to $G_{k+1}$ and solving for $G_{k+1}$ when this expression is set to nullity, we have that:
\begin{equation}
\label{eqn:Kalman Gain}
G_{k+1} = \frac{\mathcal{M}_{k}X_{k+1}}{\sigma^2 + X_{k+1}^T \mathcal{M}_{k}X_{k+1}}
\end{equation} 
Moreover, the derivation gives us an update scheme for $\mathcal{M}_{k}$ as well:
\begin{align*}
	\mathcal{M}_{k+1} &= \mathcal{M}_k - G_{k+1} X_{k+1}^T \mathcal{M}_k - \mathcal{M}_k X_{k+1} G_{k+1}^T + (\sigma^2 + X_{k+1}^T \mathcal{M}_k X_{k+1})G_{k+1} G_{k+1}^T \\
			&= \mathcal{M}_k - 2\frac{\mathcal{M}_k X_{k+1}X_{k+1}^T \mathcal{M}_k }{\sigma^2 + X_{k+1}^T \mathcal{M}_k X_{k+1}} \\
			&\quad+ \frac{(\sigma^2 + X_{k+1}^T \mathcal{M}_k X_{k+1}) \mathcal{M}_k X_{k+1}X_{k+1}^T \mathcal{M}_k }{(\sigma^2 + X_{k+1}^T \mathcal{M}_k X_{k+1})^2}
\end{align*}
To summarize, we have the optimal stochastic gradient update scheme:
\begin{equation}
\label{eqn:Parameter Estimate Update}
\beta_{k+1} = \beta_k + \frac{\mathcal{M}_{k}X_{k+1}}{\sigma^2 + X_{k+1}^T \mathcal{M}_{k}X_{k+1}}(Y_{k+1} - \beta_k^T X_{k+1})
\end{equation}
\begin{equation}
\label{eqn:Covariance Update}
\mathcal{M}_{k+1} = \left( I - \frac{\mathcal{M}_k X_{k+1} X_{k+1}^T}{\sigma^2 + X_{k+1}^T \mathcal{M}_k X_{k+1}} \right) \mathcal{M}_k
\end{equation}
Because $\sigma^2$ and $\mathcal{M}_k$ are not known, we take the kSGD method to be:
\begin{equation}
\label{eqn:updatebeta}
\beta_{k} = \beta_{k-1} + \frac{M_{k-1}X_k }{\gamma_k^2 + X_k^T M_{k-1} X_k}\left(Y_k - \beta_{k-1}^T X_k\right)
\end{equation}
\begin{equation}
\label{eqn:updateCov}
M_{k} = \left(I - \frac{M_{k-1}X_k X_k^T}{\gamma_k^2 + X_k^T M_{k-1}X_k}\right)M_{k-1}
\end{equation}
where $\beta_0 \in \mathbb{R}^n$ is arbitrary, and $M_0$ can be any positive definite matrix, but for simplicity, we will take it to be the identity. The sequence $\lbrace \gamma_k^2 \rbrace$ replace the unknown $\sigma^2$ value and will be referred to as tuning parameters. We refer to $\beta_k$ as a parameter estimate, $\mathcal{M}_k$ as the true covariance of the parameter estimate, and $M_k$ as the estimated covariance of the parameter estimate. 

kSGD's derivation raises several natural questions about the impact of the $M_k$ and $\lbrace \gamma_k^2 \rbrace$ substitutions on the behavior of kSGD:
\begin{remunerate}
\item Does the estimated covariance, $M_k$, approximate the true covariance, $\mathcal{M}_k$? Indeed, if $M_k$ approximates $\mathcal{M}_k$ then it is clear that kSGD approximates the optimal stochastic gradient estimator, and, because $\tr{\mathcal{M}_k} = \cond{\norm{\beta_k - \beta^*}_2^2}{\mathcal{F}_k}$, $M_k$ can be used as a measure of the error between $\beta_k$ and $\beta^*$. In \S \ref{subsec:Convergence of Hessian Estimate}, we prove that $M_k$ will bound $\mathcal{M}_k$ arbitrarily well from above and below in the limit up to a multiplicative constant depending on $\lbrace \gamma_k^2 \rbrace$ (Theorem \ref{thrm: M_k to 0, M_k behaves like true covariance}). Additionally, we show that $M_k \to 0$, thereby proving that $\cond{\norm{\beta_k - \beta^*}_2^2}{\mathcal{F}_k} \to 0$ (Corollary \ref{cor:Parameter Errors Converge}). Consequently, $M_k$ can be used as a stop condition; thus, kSGD addresses the first algorithmic challenge faced by proximal and SGD methods. 
\item Given that the batch minimizer converges to $\beta^*$ at a rate $\bigO{\sigma^2 n/k}$, and this convergence rate has no dependence on the conditioning of the problem (see \cite{murata1998}, Ch. 5 in \cite{van2000}), how does kSGD's convergence compare? In \S \ref{subsec:Convergence of Objective Function}, we show that kSGD's convergence rate is comparable to the batch minimizer's convergence rate, and this rate has no dependence on the conditioning of the problem (Theorem \ref{thrm:Q inv concentrates about M_k}, Corollary \ref{cor:Convergence of Objective Function}). Consequently, kSGD addresses the second algorithmic challenge faced by proximal and SGD methods.
\item Finally, what role do the tuning parameters, $\lbrace \gamma_k^2 \rbrace$, play in the convergence? In \S \ref{subsec:Conditions on Tuning Parameters}, $\lbrace \gamma_k^2 \rbrace$ are determined to play two roles: (1) $\lbrace \gamma_k^2 \rbrace$ moderate how tightly $M_k$ will bound $\mathcal{M}_k$, and (2) $\lbrace \gamma_k^2 \rbrace$ moderate how quickly $M_k \to 0$. Specifically, when $\gamma_k^2$ are within a few orders of magnitude of $\sigma^2$, the estimated covariance, $M_k$, will tightly bound the true covariance, $\mathcal{M}_k$, and when $\gamma_k^2$ are small, $M_k \to 0$ quickly. Unfortunately, if $\sigma^2$ is very large, these two tuning parameter roles conflict. This conflict is reconciled with an adaptive tuning parameter strategy motivated in Theorem \ref{thrm:M_k estimates covariance better with adaptive tuning parameter}, and constructed, in a sense, in \S \ref{subsec:Adaptive Choice of Tuning Parameters}.
\end{remunerate}

\section{Analysis}
\label{sec:Analysis} In the analysis below, we use the following conventions and notation. Recalling that $\mathcal{F}_k = \sigma(X_1,\ldots,X_k)$, we consider two types of error in our analysis:
\begin{equation}
\label{eqn:errorWeak}
e_{k} = \E{\beta_k \vert \mathcal{F}_k} - \beta^*
\end{equation}
and
\begin{equation}
\label{eqn:errorStrong}
E_k = \beta_k - \beta^*
\end{equation}
which are related by $e_k = \cond{E_k}{\mathcal{F}_k}$. 

\subsection{Convergence of the Estimated Covariance and Estimated Parameter}
\label{subsec:Convergence of Hessian Estimate}
Informally, the main result of this section, Theorem \ref{thrm: M_k to 0, M_k behaves like true covariance}, states that $M_k$ bounds $\mathcal{M}_k$ from above and below arbitrarily well in the limit. To establish Theorem \ref{thrm: M_k to 0, M_k behaves like true covariance}, we will need two basic calculations collected in Lemmas \ref{lemm:covRecursion} and \ref{lemm:errorStrongRecursion}. In the calculations below, we recall that $M_0$ is taken to be the identity for simplicity.
\begin{lemma}
\label{lemm:covRecursion}
For $j=1,\ldots,k+1$, if $ 0 < \gamma_j^{2} < \infty$ then $M_{k+1}$ is symmetric, positive definite matrices and $$
M_{k+1}^{-1} = M_{k}^{-1} + \frac{1}{\gamma_{k+1}^2}X_{k+1} X_{k+1}^T$$
\end{lemma}
\begin{proof}
By the Sherman-Morrison-Woodbury matrix identity:
$$
M_1 = I - \frac{X_1 X_1^T}{\gamma_1^2 + X_1^T X_1} \\
	= \left(I + \frac{1}{\gamma_1^2}X_1 X_1^T\right)^{-1}
$$
Hence, $M_1^{-1} = I + \frac{1}{\gamma_1^2}X_1 X_1^T$. So $M_1$ is symmetric and positive definite. Suppose this is true up to some $k$. By induction and using the Sherman-Morrison-Woodbury matrix identity, we conclude our result:
$$
M_{k+1} = M_k - \frac{M_k X_{k+1} X_{k+1}^T M_k}{\gamma_{k+1}^2 + X_{k+1}^T M_k X_{k+1}} \\
		= \left(M_{k}^{-1} + \frac{1}{\gamma_{k+1}^2}X_{k+1}X_{k+1}^T \right)^{-1}
$$
\qquad\end{proof}

\begin{lemma}
\label{lemm:errorStrongRecursion}
For $j = 1,\ldots, k+1$, if $0 < \gamma_j^2 < \infty$ then
$$
M_{k+1}^{-1}E_{k+1} = M_{k}^{-1}E_k + X_{k+1}\frac{\epsilon_{k+1}}{\gamma_{k+1}^2}
\quad \text{ and }\quad
M_{k+1}^{-1}E_{k+1} = E_0 + \sum_{j=1}^{k+1} X_{j} \frac{\epsilon_j}{\gamma_j^2}
$$
and, recalling $\mathcal{M}_k = \cond{E_k E_k^T}{\mathcal{F}_k}$ where $\mathcal{F}_k = \sigma(X_1,\ldots,X_k)$,
\begin{align*}
\mathcal{M}_{k+1} &= M_{k+1} E_0 E_0^T M_{k+1} + M_{k+1}\left( \sum_{j=1}^{k+1} \frac{\sigma^2}{\gamma_j^2} \frac{1}{\gamma_j^2}X_jX_j^T\right)M_{k+1}
\end{align*}
\end{lemma}
\begin{proof}
Using (\ref{eqn:updatebeta}) and Assumption \ref{assumption:model}:
\begin{equation}
\label{eqn:Strong Error Recursion with Noise}
E_{k+1} = \left(I - \frac{M_k X_{k+1}X_{k+1}^T}{\gamma_{k+1}^2+X_{k+1}^T M_k X_{k+1}}\right)E_k + M_k X_{k+1}\frac{\epsilon_{k+1}}{\gamma_{k+1}^2+X_{k+1}^T M_k X_{k+1}}
\end{equation}
Now, premultiplying $E_k$ by $M_k M_k^{-1}$ and using (\ref{eqn:updateCov})
\begin{align*}
E_{k+1}	&= M_{k+1}M_{k}^{-1} E_k + M_k X_{k+1}\frac{\epsilon_{k+1}}{\gamma_{k+1}^2+X_{k+1}^T M_k X_{k+1}} \\
	&= M_{k+1}\left(M_{k}^{-1} E_k + M_{k+1}^{-1} M_k X_{k+1}\frac{\epsilon_{k+1}}{\gamma_{k+1}^2+X_{k+1}^T M_k X_{k+1}} \right)
\end{align*}
Applying the recurrence relation in Lemma \ref{lemm:covRecursion}:
\begin{align*}
E_{k+1}	&= M_{k+1}\left(M_{k}^{-1} E_k + X_{k+1} \frac{\epsilon_{k+1}}{\gamma_{k+1}^2 + X_{k+1}^T M_k X_{k+1}} \right. \\
	&\quad\left. + X_{k+1} \frac{\epsilon_{k+1}}{\gamma_{k+1}^2} \frac{X_{k+1}^T M_k X_{k+1}}{ \gamma_{k+1}^2 + X_{k+1}^T M_k X_{k+1}} \right) \\
	&= M_{k+1} \left( M_{k}^{-1} E_k + X_{k+1} \frac{\epsilon_{k+1}}{\gamma_{k+1}^2} \right)
\end{align*}
Using the recursion, we have that $E_{k+1} = M_{k+1}\left( E_0 + \sum_{j=1}^{k+1} X_j \frac{\epsilon_j}{\gamma_j^2} \right)$. Thus, calculating $E_{k+1}E_{k+1}^T$, taking its conditional expectation with respect to $\mathcal{F}_k = \sigma(X_1,\ldots,X_k)$, and recalling that $\lbrace \epsilon_j \rbrace$ are independent of $\lbrace X_j \rbrace$ by Assumption \ref{assumption:model}, we establish
\begin{align*}
\mathcal{M}_k &= M_{k+1} E_0 E_0^T M_{k+1} + M_{k+1}\left( \sum_{j=1}^{k+1} \frac{\sigma^2}{\gamma_j^2} \frac{1}{\gamma_j^2}X_jX_j^T\right)M_{k+1}
\end{align*}
\qquad \end{proof}

Lemmas \ref{lemm:covRecursion} and \ref{lemm:errorStrongRecursion} suggest a natural condition on the tuning parameters in order to ensure that these calculations hold for all $k$; namely, we require for some $\delta^2$ and $\Delta^2$,
\begin{equation}
\label{eqn:Condition on Tuning Parameters}
0 < \delta^2 \leq \inf_{k} \gamma_k^2 \leq \sup_k \gamma_k^2 \leq \Delta^2 < \infty
\end{equation}
Using this condition and Lemmas \ref{lemm:covRecursion} and \ref{lemm:errorStrongRecursion}, we can also bound $\mathcal{M}_k$ by $M_k$ for all $k \in \mathbb{N}$
\begin{align*}
& M_{k+1} E_0 E_0^T M_{k+1} + \frac{\sigma^2}{\Delta^2}M_{k+1}\left( \sum_{j=1}^{k+1}  \frac{1}{\gamma_j^2}X_jX_j^T\right)M_{k+1} \preceq \mathcal{M}_{k+1} \\
&\quad\preceq  M_{k+1} E_0 E_0^T M_{k+1} + \frac{\sigma^2}{\delta^2}M_{k+1}\left( \sum_{j=1}^{k+1}  \frac{1}{\gamma_j^2}X_jX_j^T\right)M_{k+1}
\end{align*}
Using Lemma \ref{lemm:covRecursion},
\begin{align*}
& M_{k+1} E_0 E_0^T M_{k+1} + \frac{\sigma^2}{\Delta^2}M_{k+1} \left(M_{k+1}^{-1} - I \right) M_{k+1} \preceq \mathcal{M}_{k+1} \\
&\quad\preceq M_{k+1} E_0 E_0^T M_{k+1} + \frac{\sigma^2}{\delta^2}M_{k+1}\left(M_{k+1}^{-1} - I \right)M_{k+1} 
\end{align*}
Using $0 \preceq M_{k+1}E_0E_0^T M_{k+1} \preceq M_{k+1}^2 \norm{E_0}_2^2$, we have
\begin{equation}
\label{eqn:covariance bounded by estimated covariance}
\frac{\sigma^2}{\Delta^2}M_{k+1} - \frac{\sigma^2}{\Delta^2}M_{k+1}^2 
\preceq \mathcal{M}_{k+1}
\preceq  \norm{E_0}_2^2 M_{k+1}^2 + \frac{\sigma^2}{\delta^2}M_{k+1} 
\end{equation} 

Because the true covariance, $\mathcal{M}_k$, is a measure of the error between the estimated, $\beta_k$, and true parameter, $\beta^*$, then we want $\mathcal{M}_k \to 0$, which inequality (\ref{eqn:covariance bounded by estimated covariance}) suggests occurs if $M_k \to 0$. The next theorem formalizes this claim.
\begin{theorem}
\label{thrm: M_k to 0, M_k behaves like true covariance}
If Assumptions \ref{assumption:model}, \ref{assumption:X is in L2}, \ref{assumption:X spans Rn} hold, and the tuning parameters satisfy (\ref{eqn:Condition on Tuning Parameters}), then $M_k \to 0$ almost surely. Moreover, for all $\epsilon >0$, almost surely there exists a $K \in \mathbb{N}$ such that for $k \geq K$ $$ \frac{1-\epsilon}{\Delta^2} M_{k} \preceq \frac{1}{\sigma^2} \mathcal{M}_k \preceq \frac{1+\epsilon}{\delta^2} M_k$$
\end{theorem} 
\begin{remark}
There is a conventional difference between ``almost surely there exists a $K$" and ``there exists a $K$ almost surely." The first statement implies that for each outcome, $\omega$, on a probability one set there is a $K(\omega)$. The latter statement implies $\exists K$ for all $\omega$. Thus, ``almost surely there exists a $K$" is a weaker result than ``there exists a $K$ almost surely", but it is stronger than convergence in probability.
\end{remark}
\begin{proof}
To show that $M_k \to 0$, it is equivalent to prove that $\lambda_{\max}(M_k)$, the maximum eigenvalue of $M_k$, goes to $0$. This is equivalent to showing that the minimum eigenvalue of $M_k^{-1}$, $\lambda_{\min}(M_k^{-1}) = \lambda_{\max}(M_k)^{-1}$, diverges to infinity. Moreover, by Lemma \ref{lemm:covRecursion} and the Courant-Fischer Principle (Ch. 4 in \cite{courant1953}), $\lambda_{\min}(M_{k}^{-1})$ is a non-decreasing sequence. Hence, it is sufficient to show that a subsequence diverges to infinity, which we will define using the following stochastic process. Define the stochastic process $\lbrace S_k: k+1 \in \mathbb{N} \rbrace$ by $S_0 = 0$ and 
\begin{equation}
\label{eqn:renewalProcess}
S_k = \min\left\lbrace m > S_{k-1}: span[X_{S_{k-1}+1},\ldots,X_{m}] = \mathbb{R}^n \right\rbrace
\end{equation}
We will now show that the sequence $\lbrace \lambda_{\min}(M_{S_{k}}^{-1}) \rbrace$ diverges. By Lemma \ref{lemm:covRecursion},
$$ M_{S_{k+1}}^{-1} = M_{S_k}^{-1}  + \sum_{s=S_k+1}^{S_{k+1}} \frac{1}{\gamma_s^2} X_s X_s^T \succeq M_{S_k}^{-1} + \frac{1}{\Delta^2} \sum_{s=S_{k}+1}^{S_{k+1}} X_s X_s^T \succeq I + \frac{1}{\Delta^2}\sum_{j=0}^k \mathcal{X}_j $$
where
\begin{equation}
\label{eqn:bulk XX transpose}
\mathcal{X}_k = \sum_{j=S_{k}+1}^{S_{k+1}} X_j X_j^T \quad \forall k+1 \in \mathbb{N}
\end{equation}
By the Courant-Fischer Principle,
\begin{equation}
\label{eqn:Subsequence Lower Eigenvalue Bound}
\lambda_{\min}(M_{S_{k+1}}^{-1}) \geq 1 + \frac{1}{\Delta^2} \sum_{j=0}^k \lambda_{\min}(\mathcal{X}_j)
\end{equation}
Thus, we are left with showing that $\sum_{j=0}^\infty \lambda_{\min}(\mathcal{X}_j)$ diverges almost surely. To this end, we will show that $\lambda_{\min}(\mathcal{X}_j)$ will be greater than some $\alpha > 0$ infinitely often. As a result, the sum must diverge to infinity. To show this, we will use a standard Borel-Cantelli argument (see Section 2.3 in \cite{durrett2010}). Consider the events $A_j = \lbrace \lambda_{\min}(\mathcal{X}_j) \geq \alpha \rbrace$ where the choice of $\alpha$ comes from Lemma \ref{lemm:lower bound on bulk XX' expected eigenvalues}, in which $\inf_{j} \E{\lambda_{\min}(\mathcal{X}_j)} \geq \alpha > 0$. For such an $\alpha$, $\Prb{A_j} > 0$ else we would have a contradiction. By Lemma \ref{lemm:iidmathcalXk}, we have that $A_j$ are independent and that $\Prb{A_j} = \Prb{A_0} > 0$ for all $j \in \mathbb{N}$. Thus by the (Second) Borel-Cantelli Lemma (Theorem 2.3.5, \cite{durrett2010}), 
$$ \sum_{j=0}^\infty \Prb{A_j} = \infty $$
Hence, $\lambda_{\min}(\mathcal{X}_j) \geq \alpha$ infinitely often, and thus $\lambda_{\min}(M_{S_{k+1}}^{-1}) \to \infty$ as $k \to \infty$ almost surely. Now, let $\epsilon > 0$, then almost surely there is a $K \in \mathbb{N}$ such that if $k \geq K$,
$$\lambda_{\max}(M_k) \leq \min \left\lbrace \epsilon, \frac{\sigma^2 \epsilon}{\norm{E_0}^2 \delta^2} \right\rbrace$$
Applying this to inequality (\ref{eqn:covariance bounded by estimated covariance}), we can conclude the result.
\qquad \end{proof}

As a simple corollary to Theorem \ref{thrm: M_k to 0, M_k behaves like true covariance}, we prove that $E_k$ and $e_k$ converge to $0$ in some sense.
\begin{corollary}
\label{cor:Parameter Errors Converge}
If Assumptions \ref{assumption:model}, \ref{assumption:X is in L2}, and \ref{assumption:X spans Rn} hold, and the tuning parameters satisfy (\ref{eqn:Condition on Tuning Parameters}), then
\begin{remunerate}
\item $\cond{\norm{E_k}_2}{\mathcal{F}_k} \to 0$ as $k \to \infty$ almost surely.
\item $\norm{e_k}_2 \to 0$ as $k \to \infty$ almost surely.
\end{remunerate}
\end{corollary}
\begin{proof}
Note, by the Cauchy-Schwarz Inequality:
$$\cond{\norm{E_k}_2}{\mathcal{F}_k}^2 \leq \cond{\norm{E_k}_2^2}{\mathcal{F}_k} = \tr{\mathcal{M}_k}$$
By (\ref{eqn:covariance bounded by estimated covariance}) and Theorem \ref{thrm: M_k to 0, M_k behaves like true covariance}, $\tr{\mathcal{M}_k} \leq \norm{E_0}_2^2 \tr{M_k^2} + \frac{\sigma^2}{\delta^2}\tr{M_k} \to 0$ as $k \to \infty$ almost surely. By Jensen's Inequality, $\norm{e_k}_2 \leq \cond{\norm{E_k}_2}{\mathcal{F}_k}$ almost surely, thus, $\norm{e_k}_2 \to 0$ almost surely.
\qquad \end{proof}

\subsection{Insensitivity to Conditioning}
\label{subsec:Convergence of Objective Function}
The main result of this section, Theorem \ref{thrm:Q inv concentrates about M_k}, states that $M_k$ approximates a scaling of the inverse Hessian. As a consequence, Corollary \ref{cor:Convergence of Objective Function} states that kSGD's convergence rate does not depend on the conditioning of the Hessian, thereby addressing the second algorithmic challenge faced by proximal and SGD methods. Moreover, Corollary \ref{cor:Convergence of Objective Function} states that kSGD becomes arbitrarily close to the batch convergence rate in the limit (see \cite{murata1998}, Ch.5 in \cite{van2000}). 
\begin{theorem}
\label{thrm:Q inv concentrates about M_k}
If Assumptions \ref{assumption:model}, \ref{assumption:X is Linf}, and \ref{assumption:X spans Rn} hold, and $\gamma_k^2 = \gamma^2 \in (0, \infty)$ for all $k \in \mathbb{N}$, then for $\epsilon, \upsilon \in (0,1)$ the event
$$ \mathcal{E}_{k,\epsilon} = \left\lbrace \frac{\gamma^2 (1 - \epsilon)}{k} Q_*^{-1} \preceq M_k \preceq \frac{\gamma^2 (1 + \epsilon)}{k} Q_*^{-1} \right\rbrace$$ 
has probability at least $1 - \bigO{k^{-3 + \upsilon}}$.
\end{theorem}
\begin{proof}
Let $\lbrace \epsilon_k \rbrace$ be a non-negative sequence. Define $\mathcal{Y}_k = X_k X_k^T - Q_*$ and $\bar{\mathcal{Y}}_k = \frac{1}{k}\sum_{j=1}^k \mathcal{Y}_j$, and define the event
\begin{align*}
\mathcal{B}_{k} &= \left\lbrace M_{k}^{-1} \preceq \frac{k}{\gamma^2}Q_* + I(1 - \epsilon_k k) \right\rbrace \cup \left\lbrace \frac{k}{\gamma^2}Q_* + I(1 + \epsilon_k k) \preceq M_{k}^{-1} \right\rbrace \\
& = \left\lbrace \bar{\mathcal{Y}}_k \preceq -\epsilon_k \gamma^2 I \right\rbrace \cup \left\lbrace \epsilon_k \gamma^2 I \preceq \bar{\mathcal{Y}}_k \right\rbrace \tag{use Lemma \ref{lemm:covRecursion}} \\
&= \bigcap_{\norm{u}_2 = 1} \left\lbrace \left\vert u^T \bar{\mathcal{Y}}_k u \right\vert \geq \epsilon_k \gamma^2 \right\rbrace \\
&\subseteq \left\lbrace \left\vert v^T \bar{\mathcal{Y}} v \right\vert > \epsilon_k \gamma^2 \right \rbrace
\end{align*}
for an arbitrary unit vector $v \in \mathbb{R}^n$. Note that, by construction, $v^T \mathcal{Y}_j v$ are independent, mean zero random variables. By Lemma \ref{lemm:IfXLInfthenXXTisWellBounded}, $- nC^2 \leq v^T \mathcal{Y}_j v \leq nC^2$ almost surely, where $C = \norm{X}_\infty$. Therefore, by Markov's Inequality and Lemma \ref{lemm:If X L Infty and mean 0, the power of its sums are well controlled},
\begin{equation}
\label{eqn:Probability Bound on bad event}
\Prb{\mathcal{B}_k} \leq  \Prb{ \left\vert v^T \bar{\mathcal{Y}}_k v \right\vert > \epsilon_k \gamma^2} 
= \bigO{\frac{n^6C^{12}}{\epsilon_k^6 \gamma^{12} k^3}}
\end{equation}
We now turn to relating $\mathcal{E}_{k,\epsilon}$ to $\mathcal{B}_k$. Let $\tilde{M}_k = Q_*^{1/2}M_k Q_*^{1/2}$. Then,
\begin{align*}
&\mathcal{B}_k \\
 &= \left\lbrace M_{k}^{-1} \preceq \frac{k}{\gamma^2}Q_* + I(1 - \epsilon_k k) \right\rbrace \cup \left\lbrace \frac{k}{\gamma^2}Q_* + I(1 + \epsilon_k k) \preceq M_{k}^{-1} \right\rbrace \\
	&= \left\lbrace M_{k} \preceq \frac{\gamma^2}{k}\left(Q_* + I \left[ \frac{\gamma^2}{k} + \gamma^2\epsilon_k \right] \right)^{-1} \right\rbrace \cup \left\lbrace  \frac{\gamma^2}{k}\left(Q_* + I \left[ \frac{\gamma^2}{k} - \gamma^2\epsilon_k \right] \right)^{-1} \preceq M_k \right\rbrace \\
	&= \left\lbrace \tilde{M}_k \preceq \frac{\gamma^2}{k}\left(I + \gamma^2 Q_*^{-1} \left[ \frac{1}{k} + \epsilon_k \right] \right)^{-1} \right\rbrace \cup \left\lbrace  \frac{\gamma^2}{k}\left(I + \gamma^2 Q_*^{-1} \left[ \frac{1}{k} - \epsilon_k \right] \right)^{-1} \preceq \tilde{M}_k \right\rbrace \\
	&\supseteq \left\lbrace \tilde{M}_k \preceq \frac{\gamma^2}{k} \frac{\lambda_{\min}(Q_*)}{\lambda_{\min}(Q_*) + \gamma^2/k + \gamma^2 \epsilon_k} I \right\rbrace \cup \left\lbrace \frac{\gamma^2}{k} \frac{\lambda_{\max}(Q_*)}{\lambda_{\max}(Q_*) + \gamma^2/k - \gamma^2 \epsilon_k} I \preceq \tilde{M}_k \right\rbrace
\end{align*}
Now take $\epsilon_k = \frac{nC^2}{\gamma^2 k^\upsilon}$ where $\upsilon \in (0,1)$. Then, for every $\epsilon > 0$ there is a $K \in \mathbb{N}$ such that for $k \geq K$
$$ 1- \epsilon \leq \frac{\lambda_{\min}(Q_*)}{\lambda_{\min}(Q_*) + \gamma^2/k + \gamma^2 \epsilon_k} \leq  \frac{\lambda_{\max}(Q_*)}{\lambda_{\max}(Q_*) + \gamma^2/k - \gamma^2 \epsilon_k} \leq 1 + \epsilon $$
Therefore,
$$\mathcal{B}_k \supseteq \left\lbrace \tilde{M}_k \preceq \frac{\gamma^2(1-\epsilon)}{k} I \right\rbrace \cup \left\lbrace \frac{\gamma^2 (1 + \epsilon)}{k} I \preceq \tilde{M}_k \right\rbrace = \mathcal{E}_{k,\epsilon}^C$$
Using (\ref{eqn:Probability Bound on bad event}), $\Prb{\mathcal{E}_{k,\epsilon}} \geq 1 - \bigO{k^{-3+\upsilon}}$.
\qquad \end{proof}
\begin{remark}
Given Assumption \ref{assumption:X is Linf}, it is likely that with a bit more work the probability of $1 - \bigO{k^{-3 + \upsilon}}$ can be extended to some arbitrary $-A+\upsilon$ where $A \in \mathbb{N}_{> 3}$. Also, we can extend this result to a convergence in expectation by using the fact that $0 \preceq M_k \preceq I$.
\end{remark}
\begin{corollary}
\label{cor:Convergence of Objective Function}
If Assumptions \ref{assumption:model}, \ref{assumption:X is Linf}, and \ref{assumption:X spans Rn} hold, and $\gamma_k^2 = \gamma^2 \in (0,\infty)$ for all $k \in \mathbb{N}$, then for any $\epsilon,\upsilon \in (0,1)$ the event
$$ \frac{n\sigma^2(1-\epsilon)}{k} \leq \cond{D(\beta_k)}{\mathcal{F}_k} - D(\beta^*) \leq \frac{n\sigma^2(1 + \epsilon)}{k}$$
occurs with probability at least $1- \bigO{k^{-3+\upsilon}}$.
\end{corollary}
\begin{proof}
Let $\epsilon' = \epsilon/4$. On the event $\mathcal{E}_{k,\epsilon'}$, we have a uniform bound on the rate at which $M_k \to 0$. Thus, on the event $\mathcal{E}_{k,\epsilon'}$, we can strengthen Theorem \ref{thrm: M_k to 0, M_k behaves like true covariance} to $\exists K \in \mathbb{N}$ almost surely such that for any $k \geq K$ 
$$ \frac{1 - \epsilon'}{\gamma^2} M_k \preceq \frac{1}{\sigma^2} \mathcal{M}_k \preceq \frac{1 + \epsilon'}{\gamma^2} M_k $$
Combining this with Theorem \ref{thrm:Q inv concentrates about M_k}, on the event $\mathcal{E}_{k,\epsilon'}$
$$ \frac{\gamma^2 (1 - \epsilon')^2}{k}Q_*^{-1} \preceq (1-\epsilon')M_k \preceq \frac{\gamma^2}{\sigma^2} \mathcal{M}_k \preceq (1+\epsilon')M_k \preceq \frac{\gamma^2 (1+\epsilon')^2}{k}Q_*^{-1}$$
Noting, $1 - \epsilon \leq (1- \epsilon')^2 \leq (1+ \epsilon')^2 \leq 1 + \epsilon$, the event
$$ \frac{\sigma^2 ( 1- \epsilon)}{k} I \preceq Q_*^{1/2}\mathcal{M}_kQ_*^{1/2} \preceq \frac{\sigma^2 (1+\epsilon)}{k}I$$
contains $\mathcal{E}_{k,\epsilon'}$. Now, using (\ref{eqn:objective function})
\begin{align*}
\cond{D(\beta_k)}{\mathcal{F}_k} - D(\beta^*) 
 &= \cond{(\beta_k - \beta^*)^T Q_* (\beta_k - \beta^*)}{\mathcal{F}_k}  \\
 &= \tr{ Q_*^{1/2} \cond{(\beta_k - \beta^*)(\beta_k-\beta^*)^T}{\mathcal{F}_k} Q_*^{1/2}} \\
 &= \tr{ Q_*^{1/2} \mathcal{M}_k Q_*^{1/2}}
\end{align*}
Therefore, on a set containing $\mathcal{E}_{k,\epsilon'}$,
$$ \frac{n\sigma^2 ( 1- \epsilon)}{k} \leq \cond{D(\beta_k)}{\mathcal{F}_k} - D(\beta^*) \leq \frac{n \sigma^2 (1+\epsilon)}{k}$$
\qquad \end{proof}

Now that kSGD has been shown to address the two algorithmic challenges faced by proximal and SGD methods, we turn our attention to carefully analyzing the effect that the tuning parameters have on convergence.

\subsection{Conditions on Tuning Parameters}
\label{subsec:Conditions on Tuning Parameters}
The tuning parameter condition, (\ref{eqn:Condition on Tuning Parameters}), raises two natural questions:
\begin{remunerate}
\item Is the tuning parameter condition, (\ref{eqn:Condition on Tuning Parameters}), the necessary condition on tuning parameters in order to guarantee convergence?
\item Given the wide range of possible tuning parameters, is there an optimal strategy for choosing the tuning parameters?
\end{remunerate}

As will be shown in Theorem \ref{thrm:Necessary TP Condition}, the first question can be answered negatively. For example, Theorem \ref{thrm:Necessary TP Condition} suggests that if the method converges the tuning parameters could have been $\gamma_k^2 = k^p$ for $p \in (0,1]$, which is an example not covered by condition (\ref{eqn:Condition on Tuning Parameters}).
\begin{theorem}
\label{thrm:Necessary TP Condition}
Suppose $\lbrace \gamma_k^2 \rbrace$ are selected apriori. If Assumptions \ref{assumption:model}, and \ref{assumption:X spans Rn} hold, $\E{\norm{X}_2^4} < \infty$, and for any $e_0$ we have that $e_k \to 0$ then $\sum_{k=1}^\infty \gamma_k^{-2}$ diverges almost surely.
\end{theorem}
\begin{proof}
Using (\ref{eqn:Strong Error Recursion with Noise}), premultiplying $E_k$ by $M_{k}M_{k}^{-1}$, and taking conditional expectation with respect to $\mathcal{F}_{k+1}$:
$$ e_{k+1} = M_{k+1} M_k^{-1} e_k$$
Repeatedly applying the recursion, $e_{k+1} = M_{k+1} e_0$. Thus, $e_{k+1} \to 0$ for any $e_0$ is equivalent to $\lambda_{\max}(M_k) \to 0$. Note:
$$ \lambda_{\max}(M_k) = \frac{1}{\lambda_{\min}(M_k^{-1})} \geq \frac{1}{\tr{M_k^{-1}}}$$
We will prove that if $\sum_{j=1}^\infty \gamma_j^{-2} < \infty$ then the supremum of the trace of $M_{k}^{-1}$ is finite, and therefore $\lambda_{\max}(M_k) > 0$, which gives a contradiction. That is, we will show, by using Lemma \ref{lemm:covRecursion}, the supremum over all k of 
$$ \tr{M_{k}^{-1}} = \tr{ I + \sum_{j=1}^k \frac{1}{\gamma_{j}^2} X_j X_j^T} = n + \sum_{j=1}^k \frac{\norm{X_j}_2^2}{\gamma_j^2}$$
is almost surely finite. The main tool used is Kolmogorov's Three Series Theorem (Theorem 2.5.4 in \cite{durrett2010}). Suppose that $\sum_{k=1}^\infty \gamma_j^{-2} < \infty$, and let $A > 0$. By Markov's Inequality,
\begin{align*}
\sum_{j=1}^\infty \Prb{\norm{X_j}_2^2 > A \gamma_j^2} &\leq \sum_{j=1}^\infty \frac{\E{\norm{X_j}_2^2}}{\gamma_j^2 A} < \infty \\
\sum_{j=1}^\infty \frac{\E{\norm{X_j}_2^2 \1{\norm{X_j}_2^2 \leq A \gamma_j^2}}}{\gamma_j^2} &\leq \sum_{j=1}^\infty \frac{\E{\norm{X_j}_2^2}}{\gamma_j^2} < \infty \\
\sum_{j=1}^\infty \frac{ var\left(\norm{X_j}_2^2 \1{\norm{X_j}_2^2 \leq A \gamma_j^2}\right)}{\gamma_j^4} &\leq \sum_{j=1}^\infty \frac{\E{\norm{X_j}_2^4}}{\gamma_j^4} < \infty
\end{align*}
Thus, by Kolmogorov's Three Series Theorem, $\sup_{k} \tr{M_{k}^{-1}} < \infty$. Hence, $e_k \not\to 0$ and we have a contradiction.
\qquad \end{proof}
\begin{remark}
Using a condition such as $\sum_{k=1}^\infty \gamma_{k}^{-2}$ diverges or the sum over any subsequence diverges in place of (\ref{eqn:Condition on Tuning Parameters}) seems appealing. Although $\mathcal{M}_k$ can be upper bounded by $M_k$ using such a condition and Lemma \ref{lemm:errorStrongRecursion}, the main theoretical difficulty occurs in proving that the $\sum_{k=1}^\infty \Prb{A_k} = \infty$ in the proof of Theorem \ref{thrm: M_k to 0, M_k behaves like true covariance}.
\end{remark}

The second question can be understood by examining the two roles tuning parameters have in Theorem \ref{thrm: M_k to 0, M_k behaves like true covariance}. First, the tuning parameters determine how quickly $M_k^{-1}$ diverges: if the tuning parameters take on very small values, then, in light of (\ref{eqn:Subsequence Lower Eigenvalue Bound}), $\lambda_{\min}(M_k^{-1})$ will diverge quickly. Therefore, the tuning parameters should be selected to be a fixed small value when the goal is to converge quickly. Second, the tuning parameter bounds, $\delta^2$ and $\Delta^2$, determine how tightly $M_k$ bounds $\mathcal{M}_k$: if $\delta^2$ and $\Delta^2$ are close to $\sigma^2$ then $M_k$ will have a tight lower and upper bound on $\mathcal{M}_k$. Therefore, from an algorithmic perspective, if the tuning parameter bounds are close to $\sigma^2$, then $M_k$ will be a better stop condition.

In the case when $\sigma^2$ is of moderate size, the tuning parameters can be selected to be small thereby ensuring that fast convergence is achieved and that $M_k$ is a strong stop condition. On the other hand, if $\sigma^2$ is large, the tuning parameters cannot satisfy both fast convergence and ensure that $M_k$ is a strong stop condition. These opposing tuning parameter goals can be reconciled with the following result.
\begin{theorem}
\label{thrm:M_k estimates covariance better with adaptive tuning parameter}
Suppose Assumptions \ref{assumption:model}, \ref{assumption:X is in L2}, and \ref{assumption:X spans Rn} hold, and $\sigma^2 \in (\delta^2,\Delta^2)$ for some $\Delta^2 \geq \delta^2 > 0$. If there exists a sequence of tuning parameters satisfying (\ref{eqn:Condition on Tuning Parameters}) and satisfying for any $\epsilon \in (0,1)$ almost surely $\exists K \in \mathbb{N}$ such that for $k \geq K$, $|\gamma_k^2 - \sigma^2| < \epsilon \sigma^2$ then almost surely there exists a $K' \in \mathbb{N}$ such that for $k \geq K'$
$$ \mathcal{M}_k \preceq \frac{1+\epsilon}{1-\epsilon} M_k $$
\end{theorem}
\begin{proof}
Let $\epsilon \in (0,1)$. Then by assumption, almost surely there exists a $K \in \mathbb{N}$ for which $\gamma_k^2 \geq \sigma^2(1-\epsilon)$. From Lemma \ref{lemm:errorStrongRecursion} with $k \geq K$:
\begin{align*}
\mathcal{M}_{k} &= M_{k} E_0 E_0^T M_{k} + M_{k}\left( \sum_{j=1}^{k} \frac{\sigma^2}{\gamma_j^2} \frac{1}{\gamma_j^2}X_jX_j^T\right)M_{k} \\
&=M_{k}E_0 E_0^T M_{k} 
	+ M_{k}\left( \sum_{j=1}^{K-1} \left(\frac{\sigma^2}{\gamma_j^2} - 1\right)\frac{1}{\gamma_j^2} X_j X_j^T \right) M_k  \\
	&\quad + M_k\left(\sum_{j=1}^{K-1} \frac{1}{\gamma_j^2} X_j X_j^T + \sum_{j=K}^{k} \frac{\sigma^2}{\gamma_j^2} \frac{1}{\gamma_j^2} X_j X_j^T \right) M_k \\
&\preceq M_k^2 \norm{E_0}_2^2 + \left(\frac{\sigma^2}{\delta^2} - 1\right)M_{k}M_{K-1}^{-1}M_{k} \tag{$\sigma^2 \geq \delta^2$ by Assumption}\\
&\quad + M_k \left(\sum_{j=1}^{K-1} \frac{1}{\gamma_j^2} X_j X_j^T + \frac{1}{1-\epsilon}\sum_{j=K}^k \frac{1}{\gamma_j^2} X_j X_j^T \right) M_k \\
&\preceq M_k^2 \norm{E_0}_2^2 + \left(\frac{\sigma^2}{\delta^2} - 1\right)M_{k}M_{K-1}^{-1}M_{k} + \frac{1}{1-\epsilon} M_k
\end{align*}
By Theorem \ref{thrm: M_k to 0, M_k behaves like true covariance}, $M_k \to 0$ almost surely. Thus, almost surely there is a $K' \in \mathbb{N}$, which we can take larger than $K$, such that if $k \geq K'$ then
$$ M_k^2 \norm{E_0}_2^2 + \left(\frac{\sigma^2}{\delta^2} - 1\right)M_{k}M_{K-1}^{-1}M_{k} \preceq \frac{\epsilon}{1-\epsilon} M_k $$
Combining the inequalities gives the result.
\qquad \end{proof}

The construction of such a tuning parameter strategy is quite difficult because a sequence which almost surely converges to $\sigma^2$ will never be known apriori. In practice, the construction of $\gamma_k^2$ will then have to depend on the data $(X_1,Y_1),(X_2,Y_2),\ldots$; therefore, a data dependent tuning parameter strategy will introduce measurability issues in Theorem \ref{thrm:M_k estimates covariance better with adaptive tuning parameter}. Thus, we only use Theorem \ref{thrm:M_k estimates covariance better with adaptive tuning parameter} as motivation for constructing a tuning parameter strategy which estimates $\sigma^2$. The construction of such tuning parameters is the content of the next subsection.
 
\subsection{Adaptive Choice of Tuning Parameters}
\label{subsec:Adaptive Choice of Tuning Parameters}
We will define a sequence of tuning parameters $\lbrace \gamma_k^2 \rbrace$ which satisfy condition (\ref{eqn:Condition on Tuning Parameters}) and incrementally estimate $\sigma^2$ when $\sigma^2 \in (\delta^2, \Delta^2)$. This choice of tuning parameters will be determined by a two-step process:
\begin{remunerate}
\item By defining a sequence of unbounded estimators $\lbrace \xi_k^2 \rbrace$ which converge to $\sigma^2$ in some sense.
\item Then, by defining $\gamma_k^2 = \phi_\tau(\xi_k^2)$ for a $\tau \in \mathbb{N}$ where
\begin{equation}
\label{eqn:truncation}
\phi_{\tau}(x) = \tau \1{ x > \tau} + 
\tau^{-1} \1{ x < \tau^{-1}} +
x \1{ \tau^{-1} \leq x \leq \tau}
\end{equation}
to ensure that condition (\ref{eqn:Condition on Tuning Parameters}) is satisfied.
\end{remunerate}
\begin{remark}
The choice of $\phi_\tau$ imposes $\delta^2 = \tau^{-1}$ and $\Delta^2 = \tau$ in (\ref{eqn:Condition on Tuning Parameters}) The use of a single parameter, $\tau$, for the bounds is strictly a matter of convenience. Using different lower and upper bounds is completely satisfactory from a theoretical perspective and can be a strategy to negotiate the trade-offs in choosing the tuning parameters. We state this case in \S \ref{sec:Algorithms}.
\end{remark}

We will use the following general form for estimators $\lbrace \xi_k^2 \rbrace$:
\begin{equation}
\label{eqn:XiUpdateForget}
 \xi_1^2 = r_1^2 \quad\text{and}\quad \xi_{k+1}^2 = \frac{1}{k+1}f_{k+1}r_{k+1}^2 + \left(1-\frac{1}{k+1}\right) \xi_{k}^2
\end{equation}
where, the residuals, $r_k$, satisfy $r_k = Y_k - X_k^T \beta_{k-1}$, and the hyperparameters, $f_k$, are non-negative. One natural choice for the hyperparameters is $f_k = 1$ for all $k$. Indeed, such a choice has the following nice theoretical guarantee, which is a consequence of Proposition \ref{prop:convergence Xi and Gamma with Forgetting} below.
\begin{corollary}
\label{cor:convergence Xi and Gamma no Forgetting}
If Assumptions \ref{assumption:model}, \ref{assumption:X is in L2}, and \ref{assumption:X spans Rn} hold and $f_k = 1$ for all $k$ then $ \lim_{k \to \infty} \E{\left. \xi_{k}^2 \right\vert \mathcal{F}_k} = \sigma^2$ almost surely. Moreover, almost surely:
$$ \phi_{\tau}(\sigma^2) \leq \liminf_{k \to \infty} \E{\left. \gamma_k^2 \right\vert \mathcal{F}_k} \leq \limsup_{k \to \infty} \E{\left. \gamma_k^2 \right\vert \mathcal{F}_k}\leq \phi_{\tau}(\sigma^2 + \tau^{-1} )$$
\end{corollary}

Although this result calls into question why $f_k$ are even considered in (\ref{eqn:XiUpdateForget}), it is misleading: if the initial residuals, $r_1^2,r_2^2,\ldots$, deviate from $\sigma^2$ significantly, then a large number of cases must be assimilated in order for the estimator of $\xi_k^2$ to converge to $\sigma^2$. A common strategy to overcome this is to shrink the initial residuals --- similarly, introduce forgetting factors --- which converge to unity. That is, we consider a positive sequence $f_k \leq 1$ such that $f_k \to 1$. However, because $f_k$ can be designed, this procedure begs the question of how and when $f_k$ should approach $1$. Indeed, the only guidance which can be given on this choice is to make $f_k$ depend on $M_k$: once $M_k$ is sufficiently small, we have an assurance that $r_k^2$ should be faithful estimators of $\sigma^2$, and so $f_k$ should be near $1$ in this regime. Since $M_k$ are measurable with respect to $\mathcal{F}_k$, a more flexible condition is to allow $f_k$ to be measurable with respect to $\mathcal{F}_k$ as well.
\begin{remark}
Despite restricting $f_k$ to be $\mathcal{F}_k$ measurable, there are still many strategies for choosing $f_k$. For example, $f_k$ can be set using a hard or soft thresholding strategy depending on the $\tr{M_k}$. Regardless, the strategy for choosing $f_k$ should depend on the problem, and an uninformed choice is ill-advised unless a sufficient amount of data is being processed. 
\end{remark}

An additional strategy is to delay the index at which $\xi_1^2$ is calculated. To be specific, we can define a stopping time $V$ with respect to $\mathcal{F}_k$ such that $V$ is finite almost surely, and let $r_k = Y_{V+k} - X_{V+k}^T\beta_{V+k-1}$. Indeed, then the process will only be started once a specific condition has been met, such as if $\tr{M_k}$ is below a threshold. This offers more flexibility than manipulating $f_k$ alone. These considerations on $f_k$ and $V$ are collected in the following assumption.
\begin{assumption}
\label{assumption:adaptive tp}
$V$ is a stopping time with respect to $\mathcal{F}_k = \sigma(X_1,\ldots,X_k)$ such that $V$ is almost surely finite with stopped $\sigma$-algebra $\mathcal{F}_V$ (p. 156 in \cite{durrett2010}). Also, take $f_k$ to satisfy:
\begin{remunerate}
\item $f_k \in [0,1]$ and $f_k \to 1$ as $k \to \infty$ almost surely.
\item $f_k$ are $\mathcal{F}_{V+k}$ measurable.
\end{remunerate}
and let $r_k = Y_{V+k} - X_{V+k}^T \beta_{V+k - 1}$, $\xi_{k}^2$ be defined as in (\ref{eqn:XiUpdateForget}), and $\gamma_{V+k}^2 = \phi_\tau(\xi_k^2)$.
\end{assumption}

We are now ready to show that a tuning parameter strategy using Assumption \ref{assumption:adaptive tp} does indeed approximate $\sigma^2$ in the limit.

\begin{proposition}
\label{prop:convergence Xi and Gamma with Forgetting}
If Assumptions \ref{assumption:model}, \ref{assumption:X is in L2}, \ref{assumption:X spans Rn}, and \ref{assumption:adaptive tp} hold then $$\lim_{k \to \infty} \cond{ \xi_{k}^2 }{ \mathcal{F}_{V+k}} = \sigma^2$$ almost surely. Moreover, almost surely:
$$ \phi_{\tau}(\sigma^2) \leq \liminf_{k \to \infty} \E{\left. \gamma_{V+k}^2 \right\vert \mathcal{F}_{V+k}} \leq \limsup_{k \to \infty} \E{\left. \gamma_{V+k}^2 \right\vert \mathcal{F}_{V+k}}\leq \phi_{\tau}(\sigma^2 + \tau^{-1} )$$ 
\end{proposition}
\begin{proof}
Applying (\ref{eqn:XiUpdateForget}) repeatedly,
\begin{align*}
\xi_{k+1}^2 &= \sum_{l=1}^\infty \xi_{k+1}^2 \1{V = l} \\
			&=\frac{1}{k+1} \sum_{l=1}^\infty \sum_{j=1}^{k+1} f_j\left(Y_{l+j} - X_{l+j}^T \beta_{l+j-1}\right)^2 \1{V = l} \\
			&= \frac{1}{k+1} \sum_{j=1}^{k+1} f_j \sum_{l=1}^\infty\left(\epsilon_{l+j} + X_{l+j}^T(\beta^* - \beta_{l+j-1}) \right)^2 \1{V=l} \tag{Assumption \ref{assumption:model}}
\end{align*}
Taking the conditional expectations, using the Conditional Monotone Convergence Theorem (Theorem 5.1.2 in \cite{durrett2010}), and using the facts that $f_j$ and $\1{V=l}$ are measurable with respect to $\mathcal{F}_{V+k+1}$ by construction:
\begin{align*}
&\cond{\xi_{k+1}^2 }{\mathcal{F}_{V+k+1}} \\
	&= \frac{1}{k+1} \sum_{j=1}^{k+1} \sum_{l=1}^\infty \cond{\left(\epsilon_{l+j} + X_{l+j}^T(\beta^* - \beta_{l+j-1}) \right)^2 \1{V=l} }{\mathcal{F}_{V+k+1}} \\
	&= \frac{1}{k+1} \sum_{j=1}^{k+1} f_j \sum_{l=1}^\infty \1{V = l}\cond{ \epsilon_{l+j}^2 }{\mathcal{F}_{V+k+1}} + \1{V=l}X_{l+j}^T\mathcal{M}_{l+j-1}X_{l+j}\\
	&\quad + 2\1{V=l}\cond{\epsilon_{l+j}X_{l+j}^T(\beta^* - \beta_{l+j-1})}{\mathcal{F}_{V+k+1}} \\
	&= \frac{1}{k+1} \sum_{j=1}^{k+1} f_j \sum_{l=1}^\infty \1{V=l} \sigma^2 + \1{V=l} X_{l+j}^T \mathcal{M}_{l+j-1} X_{l+j} \\
	&= \frac{1}{k+1} \sum_{j=1}^{k+1} f_j \sigma^2 + f_jX_{V+j}^T \mathcal{M}_{V+j-1} X_{V+j}
\end{align*}
where we use the facts that (1) $\cond{(\beta^* - \beta_{l+j-1})(\beta^* - \beta_{l+j-1})^T}{\mathcal{F}_{V+k+1}} = \mathcal{M}_{l+j-1}$ since $\mathcal{F}_{V+j-1} \subset \mathcal{F}_{V+k-1}$ and we are restricting to the event $\lbrace V = l \rbrace$, and (2) $\epsilon_{j}$ are independent of $\mathcal{F}_k=\sigma(X_1,\ldots,X_k)$, thus, $\epsilon_j$ are independent of $\mathcal{F}_{V+k}$. Using this equality, we will establish a lower and upper bound on $\cond{\xi_{j}^2}{\mathcal{F}_{V+j}}$.

\textbf{Lower Bound.} Since $ \mathcal{M}_{V+j} \succeq 0$, we have $ \E{\left. \xi_{k+1}^2 \right\vert \mathcal{F}_{k+1}} \geq \frac{\sigma^2}{k+1}\sum_{j=1}^{k+1} f_j$. Let $\epsilon > 0$. Because $V$ is almost surely finite and $f_k \to 1$ almost surely, almost surely there is a $K \in \mathbb{N}$ such that if $k \geq K$ then $f_k \geq (1 - \epsilon)$. Therefore ,
$$ \liminf_{k \to \infty} \E{\left. \xi_{k+1}^2 \right\vert \mathcal{F}_{V+k+1}}  \geq \sigma^2 (1- \epsilon)$$ almost surely.

\textbf{Upper Bound.} For an upper bound, using the same $\epsilon > 0$ as for the lower bound, because $V$ is almost surely finite and by Theorem \ref{thrm: M_k to 0, M_k behaves like true covariance}, almost surely there is a $K \in \mathbb{N}$ such that for $k \geq K$
$$\mathcal{M}_{V+k} \preceq \tau \sigma^2 (1+\epsilon) M_{V+k} \quad\text{and}\quad \lambda_{\max}(M_{V+k}) < \frac{\epsilon}{\tau \sigma^2 (1+\epsilon) \E{\norm{X_1}^2}}$$
Since $f_k \leq 1$:
\begin{align*}
&\E{\left. \xi_{k}^2 \right\vert \mathcal{F}_{V+k}} \\
	&\leq \sigma^2 + \frac{1}{k}\sum_{j={K+1}}^{k} X_{V+j}^T\mathcal{M}_{V+j-1} X_{V+j} + \frac{1}{k}\sum_{j=1}^{K} X_{V+j}^T \mathcal{M}_{V+j-1} X_{V+j} \\
	&\leq \sigma^2 + \frac{\tau \sigma^2(1+\epsilon)}{k}\sum_{j=K+1}^{k} X_{V+j}^T M_{V+j-1} X_{V+j} + \frac{1}{k} \sum_{j=1}^{K} X_{V+j}^T \mathcal{M}_{V+j-1} X_{V+j} \\
&\leq \sigma^2 + \frac{\epsilon}{\E{\norm{X_1}^2}} \frac{1}{k} \sum_{j=K+1}^{k} \norm{X_k}^2 + \frac{1}{k} \sum_{j=1}^{K} X_{V+j}^T \mathcal{M}_{V+j-1} X_{V+j}
\end{align*}
As $k \to \infty$, the third term will vanish and the second term will converge to $\epsilon$ almost surely by the Strong Law of Large Numbers. Therefore, almost surely
$$\sigma^2(1-\epsilon) \leq \liminf_{k \to \infty} \E{\left. \xi_{k+1}^2 \right\vert \mathcal{F}_{k+1}} \leq \limsup_{k \to \infty} \E{\left. \xi_{k+1}^2 \right\vert \mathcal{F}_{k+1}} \leq \sigma^2 + \epsilon$$
Since $\epsilon > 0$ is arbitrary, it follows that the limit of the sequence exists and is equal to $\sigma^2$ almost surely.

\textbf{Bounds on $\gamma_{V+k}^2$.} Define
$$ \nu_{\tau}(x) = \begin{cases}
\tau & x > \tau \\
x & x \leq \tau
\end{cases} \quad \text{ and } \quad 
 \psi_{\tau}(x) = \begin{cases}
x & x > \tau^{-1} \\
\tau^{-1} & x \leq \tau^{-1}
\end{cases}$$
From this definition, we have that for all $x \in \mathbb{R}$, $ \nu_{\tau} \leq \phi_{\tau} \leq \psi_{\tau}$.
In the next statement, for the maximum of a set of values, we use $\vee$, and for the minimum we use $\wedge$. Applying this relationship to $\gamma_{V+k}^2$:
$$ \E{\left. \nu_{\tau}(\xi_k^2)  \right\vert \mathcal{F}_{V+k}} \vee \tau^{-1}\leq \E{\left. \gamma_{k}^2 \right\vert \mathcal{F}_{V+k}} \leq \E{\left. \psi_{\tau}(\xi_{k}^2) \right\vert \mathcal{F}_{V+k}} \wedge \tau$$
First, we consider the right hand side:
\begin{align*}
\cond{\psi_{\tau}(\xi_k^2)}{\mathcal{F}_{V+k}} 
	&= \cond{ \xi_k^2 \1{\xi_k^2 > \tau^{-1}}}{\mathcal{F}_{V+k}} 
		+ \tau^{-1} \Prb{\left. \xi_k^2 < \tau^{-1} \right\vert \mathcal{F}_{V+k}} \\
	&\leq  \cond{\xi_k^2 }{ \mathcal{F}_{V+k}} + \tau^{-1}
\end{align*}
Applying the first part of Proposition \ref{prop:convergence Xi and Gamma with Forgetting}, $ \limsup_{k \to \infty} \cond{ \psi_{\tau}(\xi_k^2)}{ \mathcal{F}_{V+k}} \leq \sigma^2 + \tau^{-1}$ almost surely. For the reverse inequality:
\begin{align*}
\cond{ \nu_{\tau}(\xi_k^2) }{ \mathcal{F}_{V+k}} 
	&= \tau \Prb{\left. \xi_k^2 > \tau \right\vert \mathcal{F}_{V+k}} 
	+ \cond{ \xi_k^2 \1{\xi_k^2 < \tau} }{ \mathcal{F}_{V+k}} \\
	&\geq \cond{ \xi_k^2  }{ \mathcal{F}_{V+k}} 
	- \cond{ \left( \xi_k^2 - \tau \right)_+ }{ \mathcal{F}_{V+k}} \\
	&\geq \cond{ \xi_k^2  }{\mathcal{F}_{V+k}} 
	- \cond{ \left( \xi_k^2 - \sigma^2\right)_+ }{ \mathcal{F}_{V+k}} 
	- \left( \sigma^2  - \tau \right)_+ \\
	&\geq \cond{ \xi_k^2 }{\mathcal{F}_{V+k}} 
	- \cond{ \left\vert \xi_k^2 - \sigma^2 }{ \right\vert \mathcal{F}_{V+k}} 
	- \left( \sigma^2  - \tau \right)_+
\end{align*}
By the first part of the result, the first term converges to $\sigma^2$, and we are left with showing that the second term converges to $0$.  
\begin{align*}
&\cond{ \left\vert \xi_{k}^2 - \sigma^2 }{ \right\vert \mathcal{F}_{V+k}} \\
	&\leq \cond{\left\vert \frac{1}{k} \sum_{j=1}^{k} f_j\epsilon_{V+j}^2 
	- \sigma^2 \right\vert}{\mathcal{F}_{V+k}} 
	+ \frac{2}{k}\cond{\sum_{j=1}^k  f_j\left\vert \epsilon_{V+j} X_{V+j}^T
	(\beta^* - \beta_{V+j-1}) \right\vert }{ \mathcal{F}_{V+k}} \\
	&\quad + \frac{1}{k}\sum_{j=1}^k f_jX_{V+j}^T \mathcal{M}_{V+j-1} X_{V+j}
\end{align*}

For the first term, we note that (1) the argument is bounded by $\frac{1}{k} \sum_{j=1}^k \epsilon_{V+j}^2 + \sigma^2$, for which using an identical conditioning argument as in the first part of the proof, gives the quantity's expectation as $2\sigma^2$, and (2) because $V$ is finite almost surely and by the strong law of large numbers, almost surely $\lim_{k \to \infty} \frac{1}{k} \sum_{j=1}^k f_j\epsilon_j^2 = \sigma^2 $. Therefore, by the conditional Dominated Convergence Theorem, the first term converges to $0$ almost surely.

For the cross term
\begin{align*}
& \frac{2}{k}\cond{\sum_{j=1}^k  f_j\left\vert \epsilon_{V+j} X_{V+j}^T
		(\beta^* - \beta_{V+j-1}) \right\vert }{ \mathcal{F}_{V+k}} \\
	&\quad \leq \frac{2}{k} \sum_{j=1}^k \cond{ |\epsilon_{V+j}| \norm{X_{V+j}}_2 \norm{\beta^* - 
		\beta_{V+j-1}}_2 }{ \mathcal{F}_{V+k}} 
		\tag{Cauchy-Schwarz} \\
	&\quad \leq \frac{2}{k} \sum_{j=1}^k \sum_{l=1}^\infty \cond{|\epsilon_{l+j}|\norm{X_{l+j}}_2 \norm{\beta_{l+j-1} - \beta^*}_2 \1{V = l} }{ \mathcal{F}_{V+k-1}}
\end{align*}
where in the last line, the expectation and summation are alternated by the conditional Monotone Convergence Theorem. Using an identical conditioning argument as in the first part of the proof, and noting $\E{|\epsilon_j|}^2 \leq \E{|\epsilon_j|^2} = \sigma^2$ by Jensen's inequality, the cross term is bounded by 
\begin{equation*}
\frac{2}{k}\cond{\sum_{j=1}^k  f_j\left\vert \epsilon_{V+j} X_{V+j}^T
		(\beta^* - \beta_{V+j-1}) \right\vert }{ \mathcal{F}_{V+k}} 
 \leq \frac{2 \sigma}{k} \sum_{j=1}^k \norm{X_{V+j}}_2 \tr{\mathcal{M}_{V+j-1}}^{1/2}
\end{equation*}
Using the same argument as for the upper bound of $\cond{\xi_k^2}{\mathcal{F}_{V+k}}$, we have that for any $\epsilon > 0$
$$ \limsup_{k \to \infty} \frac{2}{k}\E{\left. \sum_{j=1}^k  \left\vert \epsilon_j X_j^T(\beta^* - \beta_{j-1}) \right\vert \right\vert \mathcal{F}_{k}} \leq \epsilon$$
Therefore, the limit exists and is $0$. To summarize:
$$\liminf_{k \to \infty} \E{\left. \nu_{\tau}(\xi_k^2) \right\vert \mathcal{F}_k} \geq \sigma^2 - \left( \sigma^2  - \tau \right)_+ = \sigma^2 \wedge \tau$$
\qquad \end{proof}

\section{Algorithms}
\label{sec:Algorithms}
Using update equations (\ref{eqn:updatebeta}) and (\ref{eqn:updateCov}), we state a provably convergent, fast and simple algorithm for solving the linear regression problem, Algorithm \ref{alg:kSGD}, which also has a computationally fast, justified stop condition (see Theorem \ref{thrm: M_k to 0, M_k behaves like true covariance}), is insensitive to the problem's conditioning (see Theorem \ref{cor:Convergence of Objective Function}), works for a robust choice of tuning parameters (see condition (\ref{eqn:Condition on Tuning Parameters}), and Theorem \ref{thrm:Necessary TP Condition}), and can be used on very large, infinite or streaming data sources.

\begin{algorithm}[th]
	\caption{A kSGD Algorithm}
	\label{alg:kSGD}
	\KwIn{$\beta$ arbitrary initialization, rule for choosing $\gamma_k^2$, error tolerance $\epsilon > 0$}
	$M \leftarrow I$ \\
	$k \leftarrow 0$ \\
	\While{$\tr{M} > \epsilon$}{
		Read next observation: $(X,Y)$ \\
		Compute $v \leftarrow MX$ \\
		Update $\gamma_k^2$ according to user defined method\\
		Compute denominator: $s \leftarrow \gamma^2 + X^T v$ \\
		Update parameter estimate: $\beta \leftarrow \beta + v(Y- X^T \beta)/s$ \\
		Update covariance estimate: $M \leftarrow (I - \frac{vX^T}{s})M$ \\
		Increment $k \leftarrow k + 1$ \\ }
	\KwOut{Estimates: $\left(\beta,M\right)$}
\end{algorithm}

When $\sigma^2$ is very large, the stop condition can be improved by using a tuning parameter strategy which asymptotically estimates $\sigma^2$ (see Theorem \ref{thrm:M_k estimates covariance better with adaptive tuning parameter}). Motivated by this result, a class of adaptive tuning parameter strategies was constructed (\S \ref{subsec:Adaptive Choice of Tuning Parameters}) and analytically shown to converge to $\sigma^2$ in some sense (see Proposition \ref{prop:convergence Xi and Gamma with Forgetting}). One example from this class of adaptive tuning parameter strategies is stated in Algorithm \ref{alg:Adaptive Tuning Parameter, Soft Threshold}.

\begin{algorithm}[ht]
	\caption{Adaptive Tuning Parameter Soft-Threshold Sub-routine}
	\label{alg:Adaptive Tuning Parameter, Soft Threshold}
	\KwIn{counter $k$, lower tolerance $L$, upper tolerance $U$, threshold $T$, $\tr{M_k}$, estimator $\xi_{k}^2$, residual $r_{k+1}$}
	$k \leftarrow k + 1$
	Calculate $f_k \leftarrow \left[1 + \exp(\tr{M_{k-1}} - T)\right]^{-1}$ \\	
	Update $\xi_{k}^2$ by (\ref{eqn:XiUpdateForget}) \\
	$\gamma_{k}^2 \leftarrow \min\left\lbrace U, \max\lbrace L, \xi_{k}^2 \rbrace \right\rbrace$ \\
	\KwOut{$\left(\xi_{k}^2, \gamma_{k}^2, k\right)$}
\end{algorithm}

\section{Complexity Comparison}
\label{sec:complexity}
\begin{table}
\centering
\caption{A comparison of the number of data points assimilated, gradient evaluations, Hessian evaluations, floating point operations and memory requirements per iteration between SGD, kSGD, and SQN. For SQN, $M$ is the number of curvature correction pairs stored, $b$ is the size of batch gradients, $b_h$ is the size of batch Hessians, and $L$ is the number of iterations between BFGS updates.}
\label{tbl:optimization metrics}
\begin{tabular}{@{} c c c c c c @{}}\toprule
Method & Data Assim. & Gradient & Hessian & FP Ops & Memory \\ \midrule
SGD & 1 & 1 & 0 & $\bigO{n}$ & $\bigO{n}$ \\
kSGD & 1 & 1 & 0 & $\bigO{n^2}$ & $\bigO{n^2}$ \\
SQN & $b+b_h/L$ & $b$ & $b_h/L$ & $\bigO{(M+b)n+ (b_h/L)n^2}$ & $\bigO{Mn}$
\end{tabular}
\end{table}

Table \ref{tbl:optimization metrics} reports several common metrics for assessing stochastic optimization algorithms. One notable property in Table 1 is that kSGD has the highest memory requirements. Therefore, when $n$ is large, kSGD would be infeasible; this can be addressed by making kSGD a low-memory method, but this will not be considered further in this paper. Another notable property is that the floating point costs per data point assimilated appear to be approximately the same for kSGD and SQN. However, the comparability of these two values depends on $b_h$: for many statistical regression problems, $b_h < n$ would lead to rank deficient estimates of the Hessian, which would then lead to poorly scaled BFGS updates to the inverse Hessian. Therefore, $b_h$ should be taken to be larger than $n$ to ensure full rank estimates of the Hessian. In this case, SQN requires at least $\mathcal{O}(n^3/L)$ floating point operations per data point assimilated. Since $L$ is recommended to be $10$ or $20$ \cite{byrd2016}, SQN has a greater computational cost than kSGD when $n > 20$.

\section{Numerical Experiments}
\label{sec:Empirical Example}
Three problem were experimented on: a linear regression problem on medical claims payment by CMS \cite{CMSData,Rbloggers}, an additive non-parametric Haar wavelet regression problem on gas sensor readings \cite{GasData,fonollosa2015}, and a logistic regression problem on adult income classification \cite{AdultData,kohavi1996}. For each problem, the dimension of the unknown parameter ($n$), number of observations ($N$), condition number ($\kappa$) of the Hessian at the minimizer, and the optimization methods implemented on the problem are collected in Table \ref{tbl:Summary of Problems}. 

\begin{remark} 
For the linear and Haar wavelet regression problems, the Hessian does not depend on the parameter, and so it can be calculated directly. For the logistic regression problem, the minimizer was first calculated using generalized Gauss Newton (GN) \cite{wedderburn1974} and confirmed by checking that the composite gradient at the minimizer had a euclidean norm no larger than $10^{-10}$; then, the Hessian was calculated at this approximate minimizer. 
\end{remark}

\begin{table}[hb]
\centering
\caption{A tabulation of the number of parameters ($n$), number of observations ($N$), condition number ($\kappa$) of the Hessian at the minimizer, and optimization methods implemented for each of the three problem types. Note, the Haar wavelet regression problem's maximum eigenvalue was $28.7$, but its smallest eigenvalue was within numerical precision of zero.}
\begin{tabular}{@{} c c c c c @{}}\toprule
Problem & $N$ & $n$ & $\kappa$ & Methods \\ \midrule
CMS-Linear & $2,801,660$ & $34$ & $2.44\times10^6$ & kSGD,SGD,SQN \\
Gas-Haar & $4,177,004$ & $1,263$ & $-$ & kSGD, SGD, SQN\\
Income-Logistic & $30,162$ & $29$ & $1.96\times10^{24}$ & kSGD, SGD, SQN, GN
\end{tabular}
\label{tbl:Summary of Problems}
\end{table} 

For each method, intermediate parameter values, elapsed compute time, and data points assimilated (ADP) were periodically stored. Once the method terminated, the objective function was calculated at each stored parameter value using the entire data set. The methods are compared along two dimensions:
\begin{remunerate}
\item Efficiency: the number of data points assimilated (ADP) to achieve the objective function value. The higher the efficiency of a method, the less information it needs to minimize the objective function. Thus, higher efficiency methods require fewer data points or fewer epochs in order to achieve the same objective function value in comparison to a lower efficiency method.
\item Effort: the elapsed time (in seconds) to achieve the objective function value. This metric is a proxy for the cost of gradient evaluations, Hessian evaluations, floating point operations, and I/O latencies. Thus, higher effort methods require more resources or more time in order to achieve the same objective function value in comparison to a lower effort method.
\end{remunerate}

The methods are implemented in the Julia Programming Language (v0.4.5). For the linear and logistic regression problem, the methods were run on an Intel i5 (3.33GHz) CPU with 3.7 Gb of memory; for the Haar Wavelet regression problem, the methods were run on an Intel X5650 (2.67GHz) CPU with 10 Gb of memory. 

\begin{remark}
The objective function for the linear and Haar wavelet regression is the mean of the residuals squared (MRS). Therefore, for these problems, the results are discussed in terms of MRS.
\end{remark}

\subsection{Linear Regression for CMS Payment Data}
\label{subsec:linear regression}
We modeled the medical claims payment as a linear combination of the patient's sex, age and place of service. Because the explanatory variables were categorical, there were $n = 34$ parameters. The optimal MRS was determined to be $38,142.6$ using an incremental QR algorithm \cite{miller1992}. 

The three methods, SGD, kSGD and SQN, were initialized at zero. For SGD, the learning rate was taken to be of the general form
\begin{equation}
\label{eqn:SGDlearningRates}
\eta(k,p,c_1,c_2,c_3) = c_1 \1{k \leq c_2} + \frac{c_3}{(k - c_2)^p} \1{k > c_2}
\end{equation}
where $k$ is the ADP, $p \in (0.5,1]$ (see \cite{robbins1951,murata1998}), $c_1 \in [0,\infty)$, $c_2 \in [0,\infty]$ (see \cite{bert2010}), and $c_3 \in (0,\infty)$. SGD was implemented for learning rates over a grid of values for $p$, $c_1$, $c_2$ and $c_3$. The best learning rate, $(p = 0.75, c_1 = 0.01, c_2 = 10^5, c_3 = 1)$, achieved the smallest MRS. Note, this learning rate took only $0.01$ seconds longer per epoch than the fastest learning rate. For SQN, the parameters $b$, $b_h$ and $L$ were allowed to vary between the recommended values $b = 100,1000$, $b_h = 300, 600$, $L = 10, 20$ \cite{byrd2016}, $M$ was allowed to take values $10, 20, 34$, and the learning rate constant, $c$, was allowed to vary over a grid of positive numbers. The best set of parameters, $(b = 1000, b_h = 300, L = 20, M = 34, c = 2)$, came within $2\%$ of the optimal MRS with the smallest ADP and least amount of time. The tuning parameters for kSGD, summarized in Table \ref{tbl:linear_tuningParameters}, were selected to reflect the concepts discussed in \S \ref{subsec:Conditions on Tuning Parameters} and were not determined based on any results from running the method.

\begin{table}[bh]
\centering
\caption{Tuning parameter selection for kSGD method. kSGD-1 uses a tuning parameter based on Theorem \ref{thrm:Necessary TP Condition}. kSGD-2 uses a tuning parameter based on (\ref{eqn:Covariance Update}). kSGD-3 uses a tuning parameter to increase the speed of convergence based on the discussion in \S \ref{subsec:Conditions on Tuning Parameters}.}
\label{tbl:linear_tuningParameters}
\begin{tabular}{@{} c c @{}}\toprule
Label & $\gamma_k^2$ \\ \midrule
kSGD-1 & $k^{-1}$ \\
kSGD-2 & $38,000$ \\
kSGD-3 & $0.0001$
\end{tabular}
\end{table} 

Figure \ref{fig:linear_comparison} visualizes the differences between SGD, kSGD and SQN in terms of efficiency and effort. In terms of efficiency, kSGD-1, kSGD-3 and SQN are comparable, whereas kSGD-2 and SGD perform quite poorly in comparison. For kSGD-2, this behavior is to be expected for large choices of $\gamma_k^2$ as discussed below. For SGD, despite an optimal choice in the learning rate, it still does not come close to the optimal MRS after a single epoch. Even when SGD is allowed to complete multiple epochs so that its total elapsed time is greater than kSGD's single epoch elapsed time (Figure \ref{fig:linear_comparison}, right), SGD does not make meaningful improvements towards the optimal MRS. Indeed, this is to be expected as the rate of convergence of SGD is $\mathcal{O}(\sigma^2\kappa^2/k)$ where $k$ is the ADP \cite{bousquet2008}, and, therefore, SGD requires approximately $\mathcal{O}(10^9)$ epochs to converge to the minimizer. We also see in Figure \ref{fig:linear_comparison} (right) that kSGD-1 and kSGD-3 require much less effort to calculate the minimizer in comparison to SQN.  

\begin{figure}[ht]
\centering
\includegraphics[scale=0.65]{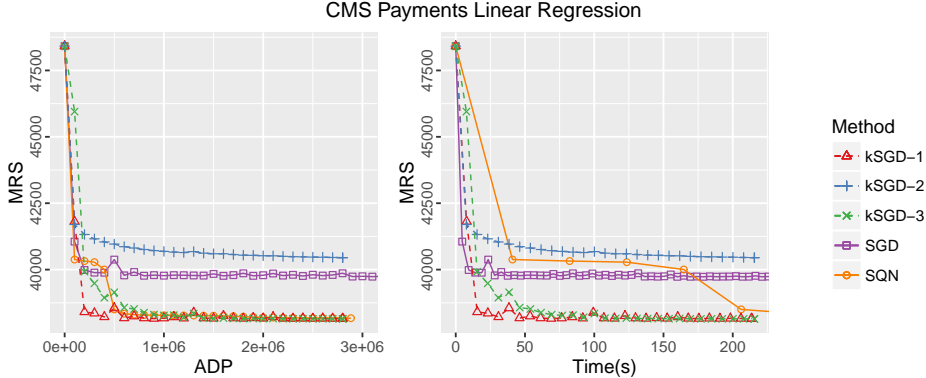}
\caption{A comparison of the performance of SGD, kSGD and SQN for the linear regression problem. {\em Left:} In terms of efficiency, kSGD-1, kSGD-3 and SQN are comparable, whereas kSGD-2 and SGD perform quite poorly in comparison. {\em Right:} kSGD-1 and kSGD-3 produce nearly optimal estimates within the first 50 seconds, which is approximately the amount of time SGD needs to complete one epoch. Also, kSGD-1 and kSGD-3 require less effort than SQN.}
\label{fig:linear_comparison}
\end{figure}

Figure \ref{fig:kSGD_Covariance} visualizes the behavior of the covariance estimates for each of the three kSGD tuning parameter choices. We highlight the sluggish behavior of $M_k$ for kSGD-2, which underscores one of the ideas discussed in \S \ref{subsec:Conditions on Tuning Parameters}: if $\sigma^2$ is large, choosing $\gamma_k^2$ to approximate $\sigma^2$ for all $k$ will slow down the convergence of the algorithm. Another important property is that, despite some variability, $\tr{M_k}$ is reflective of the decay in MRS; this property empirically reinforces the result in Theorem \ref{thrm: M_k to 0, M_k behaves like true covariance}, and the claim that $M_k$ can be used as a stop condition in practice.

\begin{figure}[ht]
\centering
\label{fig:kSGD_Covariance}
\includegraphics[scale=0.65]{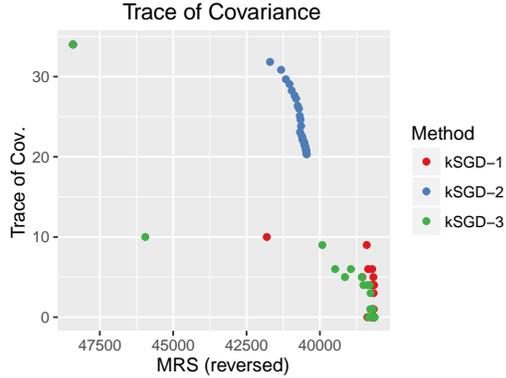}
\caption{A comparison of the MRS and trace of the covariance. The rapid decay of kSGD-1's and kSGD-3's covariance is reflected in the rapid decay of their MRS. On the other hand, kSGD-2's covariance is decaying slowly and this too is reflected in the slow decay of the MRS.}
\end{figure}

\subsection{Nonparametric Wavelet Regression on Gas Sensor Readings}
\label{subsec:nonparametric wavelet regression}
We modeled the ethylene concentration in a mixture of Methane, Ethylene and Air as an additive non-parametric function of time and sixteen gas sensor voltage readings. Because the response and explanatory variables were in bounded intervals, the function space was approximated using Haar wavelets without shifts \cite{haar1910}. The resolution for the time component was $9$, and the resolution for each gas sensor was $3$, which resulted in features and a parameter of dimension $n = 1,263$. 
\begin{remark} Because of their high-cost of calculation, the features were calculated in advance of running the methods. \end{remark}

Again, SGD, kSGD and SQN were implemented and intialized at zero. For SGD, using the same criteria as described in \S \ref{subsec:linear regression}, the best learning rate was found to be $(p = 0.8, c_1 = 0.0, c_2 = 0.0, c_3 = 1)$. For SQN, regardless of the choice of parameters (over a grid larger than the one used in \S \ref{subsec:linear regression}), the BFGS estimates quickly became unstable and caused the parameter estimate to diverge. For kSGD, the method was implemented with $\gamma_k^2 = 0.0001$.

Figure \ref{fig:haar_compare} compares SGD and kSGD. Although kSGD is much more efficient than SGD, it is remarkably slower than SGD. For this problem, this difference in effort can be reduced by using sparse matrix techniques since at most $74$ of the $1,263$ components in each feature vector are non-zero; however, for dense problems this issue can only be resolved by parallelizing the floating point operations at each iteration.

\begin{figure}[ht]
\centering
\label{fig:haar_compare}
\includegraphics[scale=0.65]{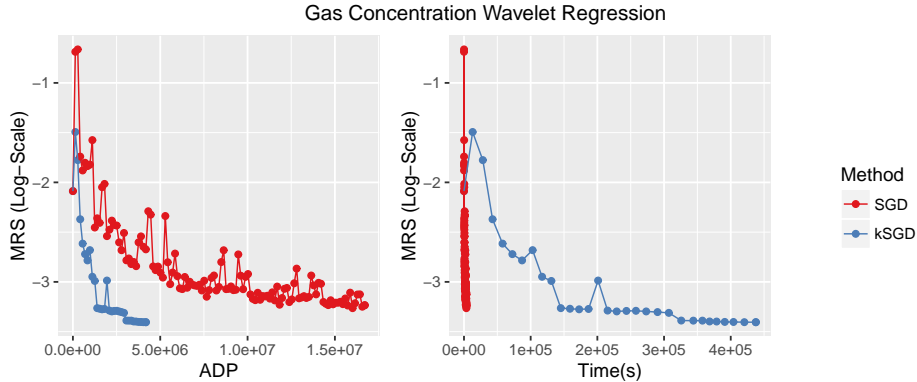}
\caption{A comparison of SGD and kSGD on the Haar wavelet regression problem. {\em Left:} kSGD is more efficient than SGD. {\em Right:} kSGD requires significantly more effort.}
\end{figure}

\subsection{Logistic Regression on Adult Incomes}
\label{subsec:logistic regression}
We modeled two income classes as a logistic model with eight demographic explanatory variables. Four of the demographic variables were continuous and the remaining four were categorical, which resulted in $n=29$ parameters.  

SGD, kSGD, SQN and GN were implemented and intialized at zero. For SGD, the best learning rate was found to be $(p=0.5,c_1=0.0,c_2=0.0,c_3=0.01)$. For SQN, the best parameter set was found to be $(b=1000,b_h=300,L=10,M=29,c=10)$. For GN, there were no tuning parameters. kSGD was adapted to solve the GN subproblems up to a specified threshold on the trace of the covariance estimate. After one subproblem was solved, the threshold was decreased by a fixed factor, which was arbitrarily selected to be 5. The threshold was arbitrarily intialized at 15, and $\gamma_k^2$ was started at $0.0001$ and increased by a factor of $10$ until the method did not fail; the method succeeded when $\gamma_k^2 = 0.1$.

Figure \ref{fig:logistic_compare} visualizes the efficiency and effort of the four methods. In general, the behavior of kSGD, SGD and SQN follow the trends in the two preceding examples. An interesting feature is that kSGD had greater efficiency and required less effort than GN. This is due to the fact that kSGD incompletely solves the GN subproblem away from the minimizer, while GN solves the subproblem exactly at each iteration. However, it is important to note that the choice of kSGD tuning parameters is not as straightforward for the logistic regression problem as it is for the linear regression problem, and appropriate choices will not be discussed further in this paper.

\begin{figure}[ht]
\centering
\label{fig:logistic_compare}
\includegraphics[scale=0.65]{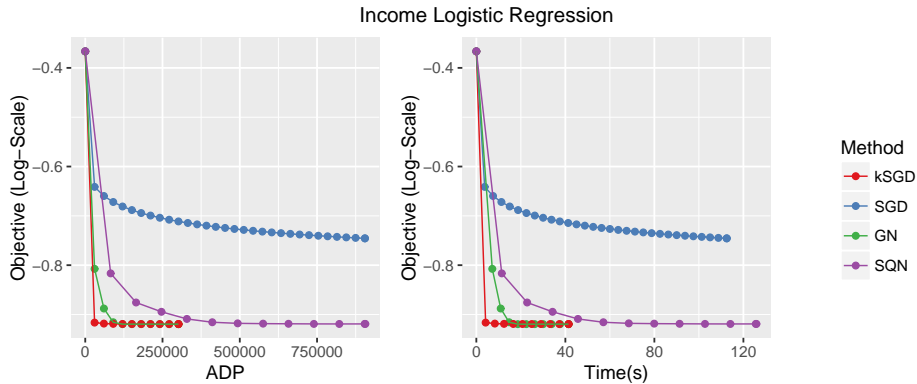}
\caption{A comparison of SGD, kSGD, SQN and GN on the logistic regression problem. {\em Left:} kSGD is more efficient than all three methods. {\em Right:} kSGD requires less effort than all three methods.}
\end{figure} 

\section{Conclusion}
\label{sec:Conclusion}
We developed and analyzed kSGD on the limited, but important class of linear regression problems. In doing our analysis, we achieve a method which (1) asymptotically recovers an optimal stochastic update method, (2) converges for a robust choice of tuning parameters, (3) is insensitive to the problem's conditioning, and (4) has a computationally efficient stop condition. As a result of our analysis, we translated this method into a simple algorithm for estimating linear regression parameters; we then successfully implemented this algorithm for solving linear, non-parametric wavelet and logistic regression problems using real data. Moreover, our analysis provides a novel strategy for analyzing the convergence of second order SGD and proximal methods, which leads to theoretical results that correspond to intuitive expectations. Finally, our analysis provides a foundation for embedding kSGD in multiple epoch algorithms, extending kSGD to a larger class of problems, and developing parallel and low memory kSGD algorithms.

\begin{appendix}
\section{Renewal Process}
We show that the stochastic process $\lbrace S_k: k+1 \in \mathbb{N} \rbrace$ defined in (\ref{eqn:renewalProcess}) is a renewal process (see Section 4.4 in \cite{durrett2010}). We define the inter-arrival times, $T_k = S_k - S_{k-1}$ for all $k \in \mathbb{N}$. 
\begin{lemma}
\label{lemm:interarrivalTimeIndependence}
If $X_1,X_2,\ldots$ are independent and identically distributed, then $T_1,T_2,\ldots$ are independent and identically distributed.
\end{lemma}
\begin{proof}
Let $k,j \in \mathbb{N}$
\begin{align*}
\Prb{T_k \geq j} &= \sum_{s=1}^\infty \Prb{T_k \geq j \vert S_{k-1} = s} \Prb{ S_{k-1} = s} \\
			  &= \sum_{s=1}^\infty \Prb{ span[X_{s+1},\ldots,X_{s+j}] = \mathbb{R}^n \vert S_{k-1} = s} \Prb{ S_{k-1} = s}
\end{align*}
Note that $\sigma(X_{s+1},\ldots,X_{s+j})$ is independent of $\sigma(X_1,\ldots,X_s)$, and so:
\begin{align*}
\Prb{T_k \geq j} &= \sum_{s=1}^\infty \Prb{span[X_{s+1},\ldots,X_{s+j}] = \mathbb{R}^n} \Prb{ S_{k-1} = s}
\end{align*}
Finally, $X_1,\ldots,X_j$ has the same distribution as $X_{s+1},\ldots,X_{s+j}$. Therefore:
\begin{align*}
\Prb{T_k \geq j} &= \sum_{s=1}^\infty \Prb{span[X_1,\ldots,X_j] = \mathbb{R}^n} \Prb{S_{k-1} = s} \\
			 & = \Prb{T_1 \geq j} \sum_{s=1}^\infty \Prb{S_{k-1} = s} \\
			 & = \Prb{T_1 \geq j}
\end{align*}
We have established that $T_1,T_2,\ldots$ are identically distributed. Now let $k_1 < k_2 < \cdots < k_r $ be positive integers with $r \in \mathbb{N}$. Let $j_1,\ldots,j_r \in \mathbb{N}$.
\begin{align*}
&\Prb{T_{k_r} = j_r,\ldots,T_{k_1} = j_1} \\
&= \sum_{s=1}^\infty \Prb{T_{k_r} = j_r \vert S_{k_r-1}=s,T_{k_{r-1}}= j_{r-1},\ldots,T_{k_1}=j_1} \\ &\quad\times \Prb{S_{k_r-1}=s,T_{k_{r-1}}= j_{r-1},\ldots,T_{k_1}=j_1}
\end{align*}
Just as above, $\sigma(X_{s+1},\ldots,X_{s+j_r})$ is independent of $\sigma(X_1,\ldots,X_s)$ and so:
\begin{align*}
&\Prb{T_{k_r} = j_r,\ldots,T_{k_1} = j_1} \\
&= \sum_{s=1}^\infty \Prb{T_{k_r} = j_r} \Prb{S_{k_r-1}=s,T_{k_{r-1}}= j_{r-1},\ldots,T_{k_1}=j_1} \\
&= \Prb{T_{k_r} = j_r} \sum_{s=1}^\infty \Prb{S_{k_r-1}=s,T_{k_{r-1}}= j_{r-1},\ldots,T_{k_1}=j_1} \\
&= \Prb{T_{k_r} = j_r} \Prb{T_{k_{r-1}}= j_{r-1},\ldots,T_{k_1}=j_1}
\end{align*}
Applying the argument recursively, we have established independence.
\qquad\end{proof}

\begin{lemma}
\label{lemm:Expected Inter Arrival Time is Finite}
If $X_1,X_2,\ldots$ are independent, identically distributed, and satisfy Assumption \ref{assumption:X spans Rn}, then $\E{T_1} < \infty$ and $\E{S_k} = k \E{T_1}$ for all $k$. 
\end{lemma}
\begin{proof}
We prove that $T_1$ is bounded by a geometric random variable and so its expectation must exist. Let $\mathcal{S}_m = span[X_1,\ldots,X_m]$. In this notation, we have that $T_1 = \inf\lbrace m > 0: \mathcal{S}_m = n \rbrace$. We will now decompose $T_1$ into $P_1,\ldots,P_n$, where $P_k = \inf\lbrace m > P_{k-1}: \dim(\mathcal{S}_m) = k \rbrace$ with $P_0 = 0$. By construction:
$$0 = P_0 < P_1 < \cdots < P_n = T_1 $$
Let $U_1,\ldots,U_n$ denote the inter-arrival times defined by $U_k = P_k - P_{k-1}$. On the event that $U_{k} = j$, we have that $\exists v \in \mathbb{R}^n$ with $\norm{v} = 1$ such that $\mathcal{S}_{P_{k-1}+j-1}$ is orthogonal to $v$, and by Assumption \ref{assumption:X spans Rn}, there is a $p = 1 - \mathbb{P}{X_1^T v = 0} > 0$. Then, $\Prb{U_k = j} \leq (1- p)^{j-1}$. Thus, $\E{U_k} < \infty$. Therefore, $\E{T_1} = \E{U_1} + \E{U_2} + \cdots \E{U_n} < \infty$. Now, by Lemma \ref{lemm:interarrivalTimeIndependence}, $T_1,\ldots,T_k$ are independent and identically distributed. Therefore, $\E{S_k} = \E{T_1} + \cdots + \E{T_k} = k \E{T_1}$.  
\qquad \end{proof}

\begin{lemma}
\label{lemm:iidmathcalXk}
If $X_1,X_2,\ldots$ are independent and identically distributed, then $\mathcal{X}_0,\mathcal{X}_1,\ldots$ defined in (\ref{eqn:bulk XX transpose}) are independent and identically distributed. In particular, the eigenvalues of $\mathcal{X}_0,\mathcal{X}_1,\ldots$ are independent and identically distributed.
\end{lemma}
\begin{proof}
Let $A$ and be a measurable set.
\begin{align*}
&\Prb{\mathcal{X}_k \in A} \\
&= \sum_{s=1}^\infty \sum_{t=1}^\infty \Prb{X_{s+1}X_{s+1}^T+\cdots+X_{s+t}X_{s+t}^T \in A \vert T_{k+1} = t, S_k =s} \\
&\quad \times \Prb{T_{k+1} = t \vert  S_{k} = s} \Prb{S_k = s} \\
&= \sum_{s=1}^\infty \sum_{t=1}^\infty \Prb{X_{1}X_1^T+\cdots+X_tX_t^T \in A \vert T_{1} = t} \Prb{T_{1} = t} \Prb{S_k =s } \\
&= \Prb{\mathcal{X}_0 \in A}
\end{align*}
Thus, $\mathcal{X}_0,\mathcal{X}_1,\ldots$ are identically distributed. By the independence of $X_1,X_2,\ldots$ and the independence of $T_1,T_2,\ldots$, which is established in Lemma \ref{lemm:interarrivalTimeIndependence}, $\mathcal{X}_0,\mathcal{X}_1,\ldots$ are independent because they are functions of independent random variables. Finally, since the eigenvalues of $\mathcal{X}_k$ can be calculated using the Courant-Fisher Principle, we have that they too are independent and identically distributed.
\qquad \end{proof}

\begin{lemma}
\label{lemm:lower bound on bulk XX' expected eigenvalues}
If $X_1,X_2,\ldots$ are independent and identically distributed, and satisfy Assumptions \ref{assumption:X is in L2} and \ref{assumption:X spans Rn}, then $\exists \alpha > 0$ such that for all $k$, $\E{\lambda_{\min}(\mathcal{X}_k)} \geq \alpha > 0$.
\end{lemma}
\begin{proof}
By Lemma \ref{lemm:iidmathcalXk}, we need only consider $\mathcal{X}_0$. Suppose there exists a $v \in \mathbb{R}^n$ with $\norm{v} = 1$ such that $\E{ v^T \mathcal{X}_0 v } = 0$. Note, by Assumption \ref{assumption:X is in L2} and Cauchy-Schwarz, the expectation is well defined. Since $\mathcal{X}_0 \succeq 0$ by construction, almost surely
$$ 0 = v^T \mathcal{X}_0 v = \sum_{s=1}^{S_1} (X_s^T v)^2 $$
Thus, $X_s^T v = 0$ for all $s = 1,\ldots,S_1$. However, by construction, Assumption \ref{assumption:X spans Rn} and Lemma \ref{lemm:Expected Inter Arrival Time is Finite}, $X_1,\ldots,X_{S_1}$ span $\mathbb{R}^n$ and $S_1 < \infty$ almost surely, hence there is an $s \leq S_1$ such that $X_s^Tv \neq 0$ almost surely. Therefore, we have a contradiction. 
\qquad \end{proof}

\section{Some Properties of $L^\infty$ Random Variables} In this section, we establish some useful properties of $L^\infty$ random variables.
\begin{lemma}
\label{lemm:IfXLInfthenXXTisWellBounded}
Suppose $X \in L^\infty$ random variable taking values in $\mathbb{R}^n$. Then, for $C = \norm{X}_\infty$, 
$$ 0 \preceq X X^T \preceq nC^2 I \quad\text{ and }\quad - nC^2I \preceq XX^T - \E{XX^T} \preceq nC^2 I$$
almost surely.
\end{lemma}
\begin{proof}
The lower bound is straightforward. For the upper bound, let $v$ be any unit vector. By Cauchy-Schwartz:
\begin{align*}
v^TX X^T  v &= (X^T v)^2 
	\leq \norm{X}_2^2 \norm{v}_2^2
	\leq nC^2 \norm{v}_2^2 
	\leq nC^2
\end{align*}
where $C = \norm{X}_{\infty}$. Thus the first set of inequalities holds almost surely. Moreover, the first set of inequalities imply $0 \preceq \E{XX^T} \preceq nC^2 I$. Hence,
\begin{align*}
- \E{XX^T} & \preceq XX^T - \E{XX^T} \preceq nC^2 I - \E{XX^T} \\
- nC^2 I &\preceq XX^T - \E{XX^T} \preceq nC^2 I
\end{align*}
\qquad \end{proof}
\begin{lemma}
\label{lemm:If X L Infty and mean 0, the power of its sums are well controlled}
Let $Z_1,Z_2,\ldots,Z_k$ be mean zero, independent random variables with $\norm{Z_1}_\infty = \cdots = \norm{Z_k}_\infty = D > 0$. Then
$$ \E{\left(\sum_{j=1}^k Z_j\right)^{6}} \leq \bigO{D^{6} k^3}$$
\end{lemma}
\begin{proof}
Note that $\E{Z_j} = 0$ and $Z_j$ are independent. Hence, in the polynomial expansion, any monomial with a term that has unity exponent is going to have a zero expectation. So, we need to count all monomials whose terms have exponents at least two in the expansion:
\begin{remunerate}
\item There are $k$ terms of the form $(Z_j)^6$.
\item There are $\NchooseK{k}{2}$ terms of the form $(Z_j)^2 (Z_i)^4$ with $\frac{6!}{4!2!} = 15$ ways of choosing the exponents.
\item There are $\NchooseK{k}{2}$ terms of the form $(Z_j)^3 (Z_i)^3$ with $\frac{6!}{3!3!} = 20$ ways of choosing the exponents.
\item There are $\NchooseK{k}{3}$ terms of the form $(Z_j)^2 (Z_i)^2 (Z_l)^2$ with $\frac{6!}{2!2!2!} = 90$ ways of choosing the exponents.
\end{remunerate}
Since $-D \leq Z_j \leq D$ almost surely by assumption, we have that:
$\E{\left( \sum_{j=1}^k Z_j \right)^6} \leq \left(k + 35k^2 + 90k^3 \right)D^6$.
\qquad \end{proof}
\end{appendix}

\section*{Acknowledgements} This work is supported by the Department of Statistics at the University of Chicago. I would also like to express my deep gratitude to Mihai Anitescu for his insightful feedback on the many drafts of this work and his generous patience and guidance throughout the research and publication process. I would also like to thank Madhukanta Patel, who, despite her illnesses and advanced age, scolded me if I momentarily stopped working on this manuscript while sitting at her bedside. I am also deeply grateful to the reviewers for their detailed readings and useful criticisms.
  
\bibliography{kSGD_linear}
\bibliographystyle{siam}
\end{document}